\documentclass[12pt]{article}
\usepackage[T1]{fontenc}
\usepackage[english]{babel}
\usepackage{amsmath,amssymb,amsthm}
\usepackage{amscd}
\usepackage{mathrsfs} 
\usepackage[shortlabels]{enumitem}
\usepackage{mathrsfs, euscript}
\usepackage{lmodern}
\usepackage[ruled]{algorithm}
\usepackage{bm}
\usepackage{graphicx}
\usepackage{algpseudocode}
\usepackage[ruled]{algorithm}
\usepackage{psfrag}
\usepackage{caption}
\usepackage{subcaption}
\usepackage{color}
\usepackage{xcolor}
\usepackage{ bbold }
\usepackage{todonotes}

\RequirePackage[colorlinks,citecolor=blue,urlcolor=blue]{hyperref}
\usepackage{cleveref}

\newtheorem{theorem}{Theorem}[section]
\newtheorem{corollary}[theorem]{Corollary}
\newtheorem{lemma}[theorem]{Lemma}

\newtheorem{proposition}[theorem]{Proposition}

\newtheorem{assumption}[theorem]{Assumption}
\theoremstyle{definition}

\newtheorem{remark}[theorem]{Remark}

\numberwithin{equation}{section}

\newcommand{\tx}[1]{\textnormal{#1}}

\newcommand{\OM}{\Omega}
\newcommand{\R}{\ensuremath \mathbb{R}}

\newcommand{\n}[1]{\|#1\|}

\newcommand{\balp}{\mathrm{A}}		

\newcommand{\N}{\ensuremath \mathbb{N}}

\newcommand{\f}{f_k}
\newcommand{\FMK}{F_k}
\newcommand{\FMG}{F_k^g}

\newcommand{\fs}{f_k^s}

\newcommand{\ef}{\varepsilon_f}
\newcommand{\eq}{\varepsilon_q}
\newcommand{\F}{\mathcal{F}}

\newcommand{\Pb}{\mathbb{P}}
\newcommand{\mo}{\mathbb{1}}
\newcommand{\me}{\mathbb{E}}
\newcommand{\trr}{\Delta}

\newcommand{\damiano}[1]{\textcolor{black}{#1}}

\newcommand{\Ve}{\mathcal{V}}
\DeclareMathOperator*{\argmin}{arg\,min}

\begin{document}
	
	\title{Stochastic trust-region and direct-search methods: A weak tail bound condition and reduced sample sizing}

\author{	
		F.~Rinaldi \thanks{Dipartimento di Matematica ``Tullio Levi-Civita'', Universit\`a
			di Padova, Italy
			(\tt{rinaldi@math.unipd.it}).}
		\and
		L. N. Vicente \thanks{Department of Industrial and Systems Engineering, Lehigh University, 200 West Packer       Avenue, Bethlehem, PA 18015-1582, USA. Support for
this author was partially provided by the Centre for Mathematics of the University of Coimbra under grant FCT/MCTES UIDB/MAT/00324/2020.}
		\and
		D. Zeffiro\thanks{Dipartimento di Matematica ``Tullio Levi-Civita'', Universit\`a di Padova, Italy
			({\tt zeffiro@math.unipd.it}).}   
	}
	
	\maketitle
	\begin{abstract}

Using tail bounds, we introduce a new probabilistic condition for function estimation in stochastic derivative-free optimization which leads to a reduction in the number of samples and eases algorithmic analyses. Moreover, we develop simple stochastic direct-search and trust-region methods for the optimization of a potentially non-smooth function whose values can only be estimated via stochastic observations. For trial points to be accepted, these algorithms require the estimated function values to yield a sufficient decrease measured in terms of a power larger than~1 of the algoritmic stepsize.

Our new tail bound condition is precisely imposed on the reduction estimate used to achieve such a sufficient decrease.
This condition allows us to select the stepsize power used for sufficient decrease in such a way to reduce the number of samples needed per iteration. 
In previous works, the number of samples necessary for global convergence
at every iteration~$k$ of this type of algorithms was $O(\Delta_{k}^{-4})$, where $\Delta_k$ is the stepsize or trust-region radius.
However, using the new tail bound condition, and 
under mild assumptions on the noise, one can prove that such a number of samples is only $O(\Delta_k^{-2 - \varepsilon})$, where $\varepsilon > 0$ can be made arbitrarily
small by selecting the power of the stepsize in the sufficient decrease test arbitrarily close to~$1$.
The global convergence properties of the stochastic direct-search and trust-region algorithms are established under the new tail bound condition.

\end{abstract}
	
	\section{Introduction}
	
	We consider the following  unconstrained optimization problem

	\begin{equation}\label{gen_prob}
		\min_{x \in \R^n} f(x),
	\end{equation}
where $f$ is locally Lipschitz continuous and possibly non-smooth with $\inf f = f^* \in \mathbb{R}$.
 We assume that the original function $f$ is not computable and that the only information available about~$f$ is given by a stochastic oracle producing an estimate $\tilde {f}(x)$ for any $x \in \R^n$. 
In some contexts,  we can assume that the estimate is a random variable parameterized by $x$, that is
$$\tilde {f}(x)=F(x, \xi),$$
with the black-box oracle given by sampling on the $\xi$ space.	 
When dealing with statistical learning problems~\cite{lan2020first}, 
the function $F(x, \xi)$ evaluates the loss of the decision rule parametrized by $x$ on
a data point~$\xi$. In simulation-based engineering applications~\cite{amaran2016simulation}, the function $F(x,\xi)$ is simply related to some noisy computable version of the original function. In this case, $\xi$ represents the random variable that induces the noise, with a classic example given by Monte Carlo simulations. 
When this random variable is exact in expected value, problem \eqref{gen_prob} turns out to be  the expected loss formulation
\begin{equation}\label{prob_zm_fv}
	\min_{x \in \R^n} \mathbb{E}_\xi[F(x,\xi)], 
\end{equation}
a case addressed in recent literature, see, e.g., \cite{larson2016stochastic, shashaani2018astro} for further details.

  \subsection{A short review of stochastic derivative-free optimization} \label{s:shortrev}
	
Although the role of derivative-free
optimization is particularly important when the black box  representing the function is somehow noisy or, in general, of a stochastic type, traditional
DFO methods have been developed primarily for deterministic functions, and only recently adapted to deal with stochastic observations (see, e.g., \cite{chen2018stochastic} for a detailed discussion on this matter). We give here a brief overview of the main results available in the literature by first focusing on \emph{model-based} strategies and then moving to \emph{direct-search} approaches. Further details on these two classes of methods can be found, e.g., in \cite{audet2017derivative,conn2009introduction}. 

In \cite{larson2016stochastic}, the  authors describe a trust-region algorithm to handle noisy objectives and prove convergence when 
$f$ is sufficiently smooth (i.e., with Lipschitz continuous gradient) and the noise is drawn independently from a distribution with zero mean and finite variance, that is they aim at solving a smooth version of problem~\eqref{prob_zm_fv}, when $\xi$ is additive noise. In the same line of research, the authors in \cite{shashaani2018astro} developed a class of derivative-free trust-region algorithms, called ASTRO-DF, for unconstrained optimization problems whose objective function has Lipschitz continuous gradient and  can only be implicitly expressed via a Monte Carlo oracle.
The authors consider again an objective with noise  drawn independently from a distribution with zero mean, finite variance and a bound on the $4v$-th moment (with $v\geq 2$), and prove the almost sure convergence of their method  when  using stochastic polynomial interpolation models. Another relevant reference in this context is given by \cite{chen2018stochastic}, where the authors 
analyze a trust-region model-based algorithm
for solving unconstrained stochastic optimization problems. They consider random models of a smooth objective function, obtained from stochastic observations of the function or its gradient. Convergence rates for this class of methods are reported in~\cite{blanchet2019convergence}.  The frameworks analyzed in~\cite{blanchet2019convergence,cartis2018global,chen2018stochastic} extend the trust-region DFO method based on probabilistic models described in~\cite{bandeira2014convergence}. It is important to notice that the randomness in the models described in~\cite{bandeira2014convergence} comes from the way  sample points are chosen, rather than from noise in the function evaluations.
All the above-mentioned model-based approaches consider functions with a certain degree of smoothness (e.g., with Lipschitz continuous gradient)
and assume that a probabilistically accurate gradient estimate  (e.g., some kind of probabilistically fully-linear
model) can be generated, while of course such an estimate is not available when dealing with non-smooth functions.

A detailed convergence rate analysis of stochastic direct-search variants is reported in~\cite{dzahini2020expected} for the smooth case, i.e., for an objective function with Lipschitz continuous gradient. 
The main theoretical results are obtained by suitably adapting the supermartingale-based framework proposed in~\cite{blanchet2019convergence}. 
A stochastic mesh adaptive direct search for black-box nonsmooth optimization is proposed in~\cite{audet2021stochastic}. The authors prove convergence with probability one to a Clarke stationary point~\cite{clarke1990optimization} of the objective function by assuming that stochastic observations are sufficiently
accurate and satisfy a variance condition. The considered analysis adapts to the direct-search gradient-free framework the theoretical analysis given in  \cite{paquette2020stochastic} for a class of stochastic gradient-based methods. It was extended in \cite{dzahini2022constrained} to the constrained case.

	 \damiano{ In a different line of work, zeroth-order methods, first analyzed in \cite{nesterov2017random} for stochastic objectives, make use of two point estimates to approximate the gradient of a smoothed version of the objective. In \cite{ghadimi2013stochastic} and \cite{nesterov2017random}, complexity bounds are given in the stochastic smooth non-convex setting and the stochastic convex non-smooth setting respectively. In \cite{lin2022gradient}, such bounds are extended to the stochastic non-smooth non-convex setting, measuring convergence with the $(\delta, \varepsilon)$-Goldstein subdifferential. For a survey of zeroth order methods with applications to machine learning problems we refer the reader to \cite{li2022stochastic}. 
		Other approaches recently adapted from the deterministic setting to stochastic derivative-free/zeroth-order  optimization include quasi-Newton methods \cite{menickelly2023stochastic}, the stochastic cubic regularized Newton \cite{roy2022stochastic}, and adaptive regularization methods with cubics \cite{scheinberg2022stochastic}, requiring stochastic estimates of both the objective gradient and also of the objective Hessian in the latter two cases.}
	
	\subsection{The contributions of this manuscript} \label{s:cp}
	
		A main goal of this manuscript is to introduce a tail bound probabilistic condition \damiano{leading to a reduced number of samples per iteration} when dealing with a stochastic black-box function in general direct-search and trust-region schemes. This probabilistic condition focuses on the \textit{reduction estimate}, that is the estimate of the difference between the function at the current iterate and at a potential next iterate, used in the acceptance test of those derivative-free algorithms. It expresses a bound on the probability that the reduction estimate error is greater than a fraction of a stepsize power characterizing the sufficient decrease needed for trial-point acceptance, and can therefore be easily adapted to different choices of the power defining such a sufficient decrease.

Our condition enables us to define a trade-off between noise, algorithm parameters, and number of samples per iteration needed to achieve global convergence, \damiano{which in this context should be intended as convergence to stationary points 
	regardless of the starting point chosen~\cite{lucidi2002global}}.
One of our results is that if all the noise moments are finite, like in the case of Gaussian noise, we only need $O(\Delta_{k}^{-2-\varepsilon})$ samples, where $\Delta_k$ is the stepsize at iteration~$k$, as described in Corollary \ref{cor:samplenum}.
Here, $\varepsilon > 0$ can be made arbitrarily small by selecting the sufficient decrease power arbitrarily close to~$1$. 
  This result compares to the $O(\Delta_k^{-4})$ number of samples required in previous works on stochastic trust-region methods \cite{blanchet2019convergence,  chen2018stochastic,shashaani2018astro} and stochastic direct-search methods \cite{audet2021stochastic, dzahini2020expected, dzahini2022constrained}, \damiano{under a finite variance assumption for the noise}. 	In those works, the sufficient decrease power is taken equal to 2, with the exception of \cite{dzahini2020expected} where the power is considered greater than 1. This article also shows that the number of samples needed can be lowered to $O(\Delta_k^{-\varepsilon})$ when the sampling errors are suitably correlated and the random number generator is known, \damiano{and in particular under a Lipschitz continuity assumption used in the analysis of zeroth-order methods}, as it is proved in Corollary~\ref{cor:correlated}. 
		
		We introduce two different algorithmic schemes, namely a simple stochastic direct-search strategy and a stochastic version of the basic deterministic trust-region scheme reported in~\cite{liuzzi2019trust}.
Both schemes work as follows: they randomly generate a direction (direct search) or a linear term (trust region); then generate the new iterate by either moving along the direction (direct search) or by solving a trust-region subproblem (trust region); finally they use a sufficient decrease acceptance test to decide if the new point can be accepted (successful iteration) or not. In this work, we use stochastic function estimates in the acceptance tests rather than exact values. Our tail bound condition applies to the function reduction estimates of both schemes, and it allows us to deduce global convergence and to take advantage of the improvement in the number of samples per iteration. \damiano{We point out that this is the first time global convergence is proved for a stochastic derivative-free trust-region algorithm for non-smooth unconstrained optimization problems.}
We also remark that the convergence analysis of our trust-region scheme is developed under a new bound on the Hessian of the quadratic model which allows us to generate non-unit linear terms, and thus 
generalizing the deterministic version given in~\cite{liuzzi2019trust}.

Lastly, we show that, for suitable choices of the algorithmic parameters, our tail bound condition is implied by the variance conditions considered in~\cite{audet2021stochastic} and by the probabilistically accurate function estimate assumption used in \cite{audet2021stochastic, chen2018stochastic, paquette2020stochastic}. 
It is also interesting to notice that the finite variance oracle usually considered in the literature (see, e.g., \cite{larson2016stochastic, shashaani2018astro}) can be replaced by a more general finite moment oracle (see Subsection~\ref{s:si} for further details) when constructing estimates satisfying our conditions. 

	\subsection{Outline of the manuscript}
 
	In Section~\ref{condition}, we introduce our tail bound probabilistic condition, prove the new bounds on the number of samples needed per iteration to satisfy the condition, and compare it to existing conditions from the literature. We then analyze the direct-search and trust-region schemes in Sections~\ref{direct} and~\ref{trm}, respectively.  In both cases, the analysis has two main steps. In the first one, we show a result that implies convergence of the stepsize/trust-region radius to zero almost surely.  In the second one, we focus on the random sequence of the unsuccessful iterations and prove, by exploiting the first result, Clarke stationarity at certain limit points. Numerical results comparing our schemes to StoMADS on a standard set of problems are reported in Section~\ref{s:numres}. Finally, we draw some conclusions and discuss some possible extensions in Section~\ref{conclusions}.
In order to improve readability and ease the comprehension, we leave some proofs and additional numerical results to an appendix.
	
	\section{A weak tail bound
		probabilistic condition for function estimation}\label{condition}
	
In order to give convergence results for our algorithms, we need to introduce a tail bound probabilistic condition on the accuracy of the function oracle. 
The stochastic quantities defined hereafter lie in a probability space $(\Pb, \OM, \F)$, with probability measure $\Pb$ and $\sigma$-algebra $\F$ containing subsets of~$\OM$ called events, which is the space of the realizations of the algorithms under analysis. Any single outcome of the sample space $\OM$	 will be denoted by $\omega$.  For a random variable $X$ defined in $\Omega$ and $A \subset \R$ we use the shorthand $\{X \in A\}$ to denote $\{\omega \ | \ X(\omega) \in A\}$. 

 Our algorithms take a step along a certain direction, which can be a direct-search direction or a trust-region step, and in both cases there is a suitable stepsize quantifying the displacement.
   \damiano{The algorithms generate a random process, as described in detail for analogous methods, e.g.,  in \cite[Section 2.2]{audet2021stochastic} and \cite[Section 3]{chen2018stochastic}}. The random quantity realizations of the process are indicated as follows.
  The random direction, the stepsize, and the current point are denoted by $G_k$, $\Delta_k$, and $X_k$, with realizations $g_k$, $\delta_k$, and $x_k$ respectively. The random estimates of $f(X_k)$ and $f(X_k+\Delta_kG_k)$ are denoted by $F_k$ and $F_k^g$, with realizations $f_k$ and $f_k^g$ respectively. In the direct-search case, the acceptance criterion will be defined as
	\begin{equation} \label{eq:ac}
	f_k - f_k^g \geq \theta \delta_k^q \, ,
\end{equation}
for some $\theta > 0$ and $q > 1$, with $\delta_k$ replaced by the norm of the step~$\|s_k\|$ in the trust-region case.
$\F_{k - 1}$ is defined as the $\sigma$-algebra of events up to the choice of~$G_k$, so that in particular, $G_k$ is always measurable with respect to $\F_{k - 1}$, \damiano{which will be considered in the proof of Theorem \ref{t:deltak0}}. \damiano{This $\sigma$-algebra will be used to formalize conditioning on the ``past history'' of the algorithm up to the choice of $G_k$.} More explicitly, $\F_{k - 1}$ is defined as the $\sigma$-algebra generated by $(F_j, F_j^g)_{j = 0}^{k-1}$ and $(G_j)_{j=0}^{k}$. $\mathbb{E}$ is used to denote expectation and conditional expectation, $\hat{v}$ as a shorthand for $v/\n{v}$, with $\hat{v} = 0$ for $v = 0$, a.s.~as a shorthand for {\it almost surely}, and $[1:p]$ to denote the integers in the interval $[1, p]$. \damiano{The starting stepsize $\Delta_0$ is assumed to be deterministic, so that in particular $\me[\Delta_0] < +\infty$, implying that the conditional expectations appearing in the rest of the article are well defined.}
	
	\subsection{The weak tail bound probabilistic condition} \label{s:wtb}
	
	We now introduce our tail bound assumption related to the acceptance criterion~\eqref{eq:ac}.

\begin{assumption} \label{ass:2}
	For some $\eq > 0$ (independent of~$k$):	
	\begin{equation} \label{eq:asp2}
		\Pb\left(|\FMK - \FMG- (f(X_k) - f(X_k + \Delta_k G_k))| \geq \alpha \Delta_{k}^{q} \ | \F_{k - 1} \right) \leq \frac{\eq}{\alpha^{q/(q - 1)}}
	\end{equation}
a.s.~for every $\alpha > 0$.
\end{assumption}
The above assumption is in particular a power law \cite{virkar2012power} tail bound with exponent $q/(q - 1) + 1$.
Notice that an error bound is only assumed for the estimate of the difference $f(X_k) - f(X_k + \Delta_kG_k)$ and not for the estimates of $f(X_k)$ and $f(X_k + \Delta_k G_k)$ taken individually; basically, this bounds  the probability that the error in that estimate is large, as such an estimation plays a crucial role in the acceptance tests of the algorithms of this work.  It will be clear from Sections~\ref{s:cads} and~\ref{s:catr} that the knowledge of an upper bound on $\eq$ is needed in order to ensure convergence in the proposed algorithms. 

\begin{remark} \label{rem:numsamples}
	As described in Section \ref{s:si}, Assumption \ref{ass:2} can be made for any $q$, if the $r$-th moment of the evaluation noise is finite, for $r = q/(q - 1)$. Furthermore, for $q \in (1, 2]$, the number of samples needed to satisfy Assumption \ref{ass:2} is just $O(\Delta_k^{-2q})$ rather than the standard $O(\Delta_k^{-4})$ required under finite variance assumptions \cite{audet2021stochastic} \damiano{with exponent 2 in the sufficient decrease condition \eqref{eq:ac}}. This improvement is possible thanks to the relation between the tail bound \eqref{eq:asp2} and the acceptance criterion \eqref{eq:ac}, together with classic results from probability theory on the convergence rate for the law of large numbers. More precisely, this property will be used: for an average $A$ of $m$ i.i.d.~samples with finite $r$-th finite moment, there is a tail bound of the form $\Pb(A \geq \alpha) \leq K_{m, r}/\alpha^r \, $ with $K_{m, r} \propto m^{-\frac{r}{2}}$,  as a consequence of Rosenthal's inequality \cite{ibragimov2002exact} \damiano{and where $\propto$ stands for "proportional to". Details about these inequalities will be discussed in the appendix.}
\end{remark}

For convergence purposes, a variant of Assumption 2.1 where the real number $\alpha$ is replaced with a $\F_{k-1}$-measurable random variable $\balp$ will be needed. This is justified by the following lemma.

\begin{lemma} \label{lem:condalpha}
	Let $\balp$ be a nonnegative $\F_{k - 1}$ measurable random variable. If \eqref{eq:asp2} holds, then
	\begin{equation} \label{eq:asp2bis}
		\Pb\left(|\FMK - \FMG- (f(X_k) - f(X_k + \Delta_k G_k))| \geq \balp \Delta_{k}^{q} \ | \F_{k - 1} \right) \leq \varepsilon(\balp) := +\infty \mo_{\{0\}} +	\frac{\eq}{\balp^{q/(q - 1)}} \mo_{(0, +\infty)}
		\end{equation}
\end{lemma}   

\begin{proof}
	Let $Y = |\FMK - \FMG- (f(X_k) - f(X_k + \Delta_k G_k))| / \Delta_{k}^{q} $, and $r = \frac{q}{q - 1}$. We prove that in this case that for every $F \in \F_{k - 1}$:
	\begin{equation}\label{eq:intermediate2}
		\me\left[\mo_F \mo_{\{Y \geq \balp\}}\right] \leq \me \left[\mo_F\varepsilon(\balp)\right] \, .
	\end{equation}
	We prove this intermediate result in the case where $\balp$ is a discrete random variable with a countable set of possible realizations $\{a_i\}_{i \in \mathbb{N}}$, and then extend the result to the general case by approximation.
	Indeed we have
	\begin{equation}\label{eq:chainineq}
		\begin{aligned}
			&	\me\left[\mo_F \mo_{\{Y \geq \balp\}}\right] = \sum_{i \in \mathbb{N}} \me\left[\mo_F \mo_{\{Y \geq \balp\}} \mo_{\{\balp = a_i\}}\right] =  \sum_{i \in \mathbb{N}} \me\left[\mo_{F \cap \{\balp = a_i\}} \mo_{\{Y \geq a_i\}}\right] \\ 
			&	\leq \sum_{i \in \mathbb{N}} \me\left[\mo_{F \cap \{\balp = a_i\}}\varepsilon(a_i)\right] = \sum_{i \in \mathbb{N}} \me\left[\mo_{F} \mo_{\{\balp = a_i\}} \varepsilon(\balp)\right] = \me \left[\mo_F \varepsilon(\balp)\right]			
		\end{aligned}
	\end{equation}
	as desired, where we used that $F \cap \{\balp = a_i\}$ is measurable w.r.t. $\F_{k - 1}$ together with \eqref{eq:asp2} for $\alpha = a_i$ in the inequality. Notice that if $a_i = 0$ then by assumption $\varepsilon(a_i) = + \infty$ so the inequality is trivial.
 
	Let now $A$ be a general positive random variable, and $\{A_i\}_{i \in \mathbb{N}}$ be a decreasing sequence of discrete random variables converging to $A$ (e.g., $A_i = \sum_{j = 0}^{+ \infty} \mo_{A\in [j/2^i, (j+ 1)/2^i )} \frac{j+1}{2^i}$). Then $\{\frac{\eq}{A_i^r}\}$ is non decreasing and converges a.s.~to $\varepsilon(\balp)$, so we have all the assumptions needed to apply Beppo Levi's Lemma and get
	\begin{equation} \label{eq:beppolevi}
		\lim_{i \rightarrow \infty} \me\left[\mo_F \frac{\varepsilon_q}{\balp_i^r}\right] = \me\left[\mo_F \varepsilon(\balp)\right] \, .
	\end{equation}
	Therefore
	\begin{equation} \label{eq:fatou}
			 \me\left[\mo_F \mo_{\{Y \geq \balp\}}\right] \; = \; \lim_{i \rightarrow \infty} \me\left[\mo_F \mo_{\{Y \geq \balp_i\}}\right]
			\;  \leq \; \lim_{i \rightarrow \infty} \me [\mo_F \varepsilon(\balp_i)] \; = \; \me\left[\mo_F \varepsilon(\balp)\right] \, ,
	\end{equation}
	where we used the dominated convergence theorem in the first equality, \eqref{eq:chainineq} in the first inequality, and \eqref{eq:beppolevi} in the second equality. We have thus proved \eqref{eq:intermediate2} in the general case. Now, let $Z = \Pb(\mo_{Y\geq A} \ | \ \F_{k - 1}) = \me\left[\mo_{Y\geq A} \ | \ \F_{k - 1}\right]$. We have, for every $F\in \F_{k - 1}$,
	\begin{equation}\label{eq:za}
		\me\left[Z \mo_{F}\right] = \me\left[\mo_{Y\geq A} \mo_{F}\right] \leq \me\left[\varepsilon(\balp) \mo_F\right] \, ,
	\end{equation} 
	where the first equality follows by definition of conditional expectation and the inequality follows by \eqref{eq:intermediate2}. Since both $Z$ and $\frac{\eq}{\balp^r}$ are $\F_{k - 1}$ measurable, from \eqref{eq:za} we get $Z \leq \frac{\eq}{\balp^r}$ a.s.~as desired.  
\end{proof}   

The proof is technical and can be found in the Appendix.

In the remaining of this section, we will report the bounds on the number of samples needed to satisfy Assumption~\ref{ass:2}, as well as a comparison with existing conditions. The proofs are rather technical and can be found in the appendix. 

\subsection{Sampling improvement under the new condition} \label{s:si}

We will show that our tail bound condition can be satisfied under a reduced number of function samples.

We deal first with the case where the error of the oracle has finite $r$-th moment, for some $r > 1$:
\begin{equation} \label{eq:mbound}
	f(x) = \ \me_{\xi}[F(x, \xi)] \, ,   \quad \quad \tx{E}_{\xi}\left[|F(x, \xi) - f(x)|^r\right] \leq  \ M_r < +\infty \, . 
\end{equation}
Recall that finite $r$-th moment implies finite $r'$-th moment for any $r'\in (1, r]$. Thus for $r < 2$ assumption \eqref{eq:mbound} is weaker than assuming finite variance, while for $r > 2$ \eqref{eq:mbound} is stronger than assuming finite variance. The next result describes the number of samples needed asymptotically to satisfy the tail bound conditions as a function of $r$.
\begin{theorem} \label{th:a1}
	Assume that \eqref{eq:mbound} holds with $r = \frac{q}{q - 1}$. If $q > 2$, then Assumption \ref{ass:2} can be satisfied using $O(\Delta_k^{-q^2})$ i.i.d.~samples, while if $q \in (1, 2]$, it can be satisfied using $O(\Delta_k^{-2q})$ i.i.d.~samples.
\end{theorem}

We thus have the following corollary illustrating an improvement on the number of samples per iteration with respect to the finite variance case.

\begin{corollary}\label{cor:samplenum} 
	Let $\varepsilon \in (0, 2]$. Then, for $q = 1 + \varepsilon/2$,  $O(\Delta_k^{-2 -\varepsilon})$ samples are sufficient to satisfy Assumption \ref{ass:2}, under the finite moment assumption \eqref{eq:mbound} for $r = \frac{q}{q - 1}$.
\end{corollary}

In the rest of this section we assume that the objective is given in the form \eqref{prob_zm_fv}, and that the CRN (common number generator) framework can be applied, that is different $x$ can be sampled with fixed $\xi$. Let now $\bar{F}(x, \xi) = F(x, \xi) - f(x)$ be the sampling error. The sampling errors of close points are assumed to be correlated in the following way:
\begin{equation} \label{ass:corr}
	\me_{\xi}\left[|\bar{F}(x, \xi) - \bar{F}(y, \xi)|^r\right] \leq D_r\n{x - y}^r
\end{equation}
for some $D_r > 0$.
\damiano{First, we prove that \eqref{ass:corr} is satisfied if $F(\cdot, \xi)$ is Lipschitz continuous, uniformly in $\xi$. We remark that uniform Lipschitz continuity assumptions analogous to the one made here are standard in the analysis of zeroth-order methods \cite{lin2022gradient,nesterov2017random}. In the finite sum setting, this assumption is equivalent to the Lipschitz continuity of every summand.
	\begin{proposition} \label{p:Lips}
		Assume that $|F(x, \xi) - F(y, \xi)| \leq L_f \n{x - y}$ for every $\xi$, and for every $x, y \in \R^n$. Then \eqref{ass:corr} holds for every $r$, with $D_r = 2^r L_f^r$.
	\end{proposition}
	\begin{proof}
		Notice that from \eqref{prob_zm_fv} and the uniform $L_f$ Lipschitz continuity assumption it follows that $f$ is $L_f$~Lipschitz continuous as well. Hence, we can write
		\begin{equation*}
			\begin{aligned}
				& |\bar{F}(x, \xi) - \bar{F}(y, \xi)| = |F(x, \xi) - F(y, \xi) + (f(y) - f(x))|  \\
				& \leq  |f(x) - f(y)| + |F(x, \xi) - F(y, \xi)| \leq 2 L_f \n{x - y} \, ,
			\end{aligned}
		\end{equation*}
		and conclude    
		\begin{equation*}
			\me_{\xi}\left[|\bar{F}(x, \xi) - \bar{F}(y, \xi)|^r\right] \leq 
			\me_{\xi}\left[2^r L_f^r \n{x - y}^r\right] = 2^r L_f^r \n{x - y}^r \, ,
		\end{equation*}
		as desired.
\end{proof}}

We now present another example where \eqref{ass:corr} is satisfied, with the noise modelled as a Gaussian process, as is common practice in Bayesian optimization (see, e.g., \cite{shahriari2015taking}).

\begin{proposition}\label{p:gaussian}
	Assume that $\{F (x, \xi)\}$ is a Gaussian process with expectation $f(x)$, exponentiated kernel with amplitude $\sigma > 0$ and lengthscale $l > 0$, so that in particular
	\begin{equation} \label{eq:quadratickernel}
		\tx{Cov}_{\xi}(F(x, \xi), F(y, \xi)) = \sigma^2\tx{exp}\left( -\frac{\n{x - y}^2}{2l^2}\right)
	\end{equation}
	for every $x, y \in \R^n$. Then assumption \eqref{ass:corr} is satisfied for every $r \geq 2$ (with $D_r$ depending on $r$).
\end{proposition}

We now show how the bound given in Theorem \ref{th:a1} improves under \eqref{ass:corr}, for $r \geq 2$. 

\begin{theorem} \label{th:corrbound}
	If the random number generator is known and \eqref{ass:corr} holds with $r = \frac{q}{q - 1}$, then Assumption \ref{ass:2} can be satisfied for $q \in (1, 2]$ using $O(\Delta_{k}^{2-2q})$ i.i.d.~samples.
\end{theorem}

As a corollary we can state a further improvement in samples per iteration with respect to Corollary~\ref{cor:samplenum}.

\begin{corollary}\label{cor:correlated}	
	If $q = 1 + \frac{\varepsilon}{2}$ with $\varepsilon \in (0, 2]$ then $O(\Delta_{k}^{-\varepsilon})$ samples are sufficient to satisfy Assumption \ref{ass:2} under~\eqref{ass:corr} for $r = \frac{q}{q - 1}$.
\end{corollary}

	\subsection{Comparison with existing conditions}

In this subsection, we compare our condition with others found in the literature. 
We will start by showing that our condition is weaker than the ones imposed in~\cite{audet2021stochastic}. More precisely, it is implied by \cite[Equation (2)]{audet2021stochastic}, rewritten in our notation as
\begin{equation} \label{eq:c1}
	\begin{aligned}
		\me\left[|\FMG- f(X_k + \Delta_k G_k)|^2 \ | \ \F_{k - 1}\right] & \leq k_f^2 \Delta_k^4 \\
		\me\left[|\FMK - f(X_k)|^2 \ | \ \F_{k - 1}\right] & \leq k_f^2 \Delta_k^4 \, ,
	\end{aligned}
\end{equation}
for a constant $k_f > 0$. 
The $k_f$-variance condition in \eqref{eq:c1} is a gradient-free version of \cite[Assumption 2.4, (iii)]{paquette2020stochastic}, and more precisely can be obtained from the latter by removing the gradient related terms in the right-hand side. It is important to note here that in \cite{paquette2020stochastic} as well as in other works on smooth stochastic derivative free optimization (see, e.g., \cite{chen2018stochastic, larson2016stochastic, shashaani2018astro} and references therein), a probabilistically accurate gradient estimate is also used, while of course such an estimate is not available in a possibly non-smooth setting.

\begin{proposition} \label{p:2eq}
	Condition \eqref{eq:c1} implies Assumption \ref{ass:2} for $\eq = 4k_f^2$ and $q = 2$.
\end{proposition}

The proof of the above result relies on the conditional Chebyshev's inequality (see the proof in the appendix for details).

\begin{remark}
	In the algorithm proposed in 
	\cite{audet2021stochastic} 
	the direct-search direction at iteration $k$ is chosen before the computation of the function estimates used in the acceptance test. Thus our analysis can also be extended to that algorithm.
\end{remark}

We now describe the relation between our assumption and the $\beta$-probabilistic accuracy assumption
\begin{equation} \label{betaprob}
	\Pb\left(\{ |\FMK  - f(X_k) |\leq \tau_f \Delta_{k}^2 \}\cap \{ |\FMG - f(X_k + \Delta_k G_k)| \leq \tau_f \Delta_{k}^2| \} \ | \  \F_{k - 1}\right) \geq \beta \, ,
\end{equation}
used in \cite{audet2021stochastic, chen2018stochastic, paquette2020stochastic} in combination with other assumptions. In particular, conditions \eqref{eq:c1} are used in \cite{audet2021stochastic} and \cite{paquette2020stochastic} (as discussed above), and a probabilistic assumption on the accuracy of random models for the objective is considered  in~\cite{chen2018stochastic}. 

We show that if \eqref{betaprob} is satisfied for every $\beta$ in a certain interval, with $\tau_f$ depending on~$\beta$ and an accuracy parameter~$\varepsilon$, then also our assumption is satisfied with $\eq$ dependent on~$\varepsilon$.

\begin{proposition} \label{p:2eq2}
	Let $\varepsilon > 0$ and $\bar{p} \in (0, 1)$.
	Assume that  \eqref{betaprob} holds for every $\beta \in [1 - \bar{p}, 1)$, with $\tau_f = \tau_f(\beta)  <  \frac{1}{2}\sqrt{\frac{\varepsilon}{1 - \beta}}$. Then Assumption \ref{ass:2} holds with $\eq= \frac{\varepsilon}{\bar{p}}$ and $q=2$.
\end{proposition}

The proposition above follows from the inclusion 
	\begin{equation} \label{eq:pineq}
		\begin{aligned}
			& \{|\FMK - \FMG- (f(X_k) - f(X_k + \Delta_k G_k))| <  \alpha \Delta_k^2 \} \\ 
			&	\supset  \{|\FMK  - f(X_k)| \leq \tau_f(\beta) \Delta_k^2 \} \cap \{|\FMG- f(X_k + \Delta_k G_k) | \leq \tau_f(\beta) \Delta_k^2\}	\, ,		
		\end{aligned}
	\end{equation}
whenever $\tau_f(\beta) < \frac{\alpha}{2}$. A detailed proof is presented in the appendix.

	\section{A simple direct-search method for stochastic non-smooth functions}\label{direct}

	In this section, we first describe a simple stochastic direct-search algorithm for the unconstrained minimization problem given in \eqref{gen_prob}, where~$f$ is possibly non-smooth, and then analyze its convergence. 
	
	\subsection{The stochastic direct-search scheme} \label{s:drscheme}
	
	A detailed description of our stochastic direct-search method is given in Algorithm~\ref{alg:GS}. At each iteration, we generate a direction $g_k$ in the unit sphere (independently of the estimates of the objective function generated so far; see Step 3), and perform a step along the direction~$g_k$ with stepsize $\delta_{k}$. 
Then, at Step~4, we compute~$f_k^g$ and~$f_k$, that is the estimate values of the function at the resulting trial point~$x_k+\delta_k g_k$ and also at~$x_k$. We then accept or reject the trial point based on a sufficient decrease condition, imposing that the improvement on the objective estimate at the trial point is at least~$\theta \delta_k^q$.
If the sufficient decrease condition is satisfied, we have a successful iteration. We hence update our iterate~$x_{k+1}$ by setting it equal to the trial point and expand or keep the same stepsize at Step~5.
Otherwise, the iteration is unsuccessful, so we do not update the current solution, that is, $x_{k+1}=x_k$, but shrink the stepsize (see Step~6). 
	
	{\linespread{1.3}
		\begin{algorithm}[H]
			\caption{\texttt{Stochastic direct search}}
			\label{alg:GS}
			\begin{algorithmic}
				\par\vspace*{0.1cm}
				\item$\,\,\,1$\hspace*{0.1truecm} \textbf{Initialization.} Choose a point $x_0$, $\delta_0, \theta > 0$, $\tau \in (0, 1)$, $\bar{\tau} \in [1, 1 + \tau]$.
				\item$\,\,\,2$\hspace*{0.1truecm} \textbf{For} $k=0, 1\ldots$
				\item$\,\,\,3$\hspace*{0.9truecm} Select a direction $g_k$ in the unit sphere.
				\item$\,\,\,4$\hspace*{0.9truecm} Compute estimates $f_k$ and $f_k^g$ for $f$ at $x_k$ and $x_k + \delta_k g_k$.	
				\item$\,\,\,5$\hspace*{0.9truecm} \textbf{If} $f_k - f_k^g \ge \theta \delta_k^q$, \textbf{Then} set {\tt SUCCESS} $=$ {\tt true}, $x_{k + 1} = x_k + \delta_k g_k$, $\delta_{k + 1} = \bar{\tau}\delta_k$.
				\item$\,\,\,6$\hspace*{0.9truecm} \textbf{Else} set  {\tt SUCCESS} $=$ {\tt false}, $x_{k + 1} = x_k$, $\delta_{k + 1} = (1 - \tau)\delta_k$.
				\item$\,\,\,7$\hspace*{0.9truecm} \textbf{End if}
				\item$\,\,\,8$\hspace*{0.1truecm} \textbf{End for}
				\par\vspace*{0.1cm}
			\end{algorithmic}
		\end{algorithm}
	}
	In order for the method to convergence to Clarke stationary points, 
	 certain subsequences of $\{g_k\}$ must be dense in the unit sphere as described in Theorem \ref{t:cs}. As a remark, a dense sequence in the unit sphere can be generated using a suitable quasirandom sequence \cite{halton1960efficiency, liuzzi2019trust}.
	
	\subsection{Convergence analysis under the tail bound probabilistic condition} \label{s:cads}
	\newcommand{\tqp}{\tau_q^+}
\newcommand{\tqm}{\tau_q^-}
\newcommand{\tqdel}{\bar{\tau}_q}	
	The following theorem, which implies that the stepsize sequence $\{\Delta_k\}$ converges to zero almost surely, is a key result in the convergence analysis. In the proof, Assumption \ref{ass:2} makes it possible to unify the argument for unsuccessful and successful steps. 
 
We define now for convenience the positive constants $\tau_q^+ = (1 + \tau)^q - 1$, $\tau_q^- = 1 - (1 - \tau)^q$, and $\tqdel = \tau_q^+ + \tau_q^-$.
To obtain our result we need the following lower bound on the parameter $\theta$ defining the sufficient decrease condition, dependent on the stepsize update parameter $\tau$ and the tail bound parameter $\eq$:
\begin{equation} \label{as:mu}
	\theta >	\frac{\sqrt[r(q)]{\eq}\tqdel}{\tqm} \, ,
\end{equation}
with $r(q) = \frac{q}{q - 1}$. 
Notice that since $\tau \in (0, 1)$ we must always have $\theta > 0$. The bound~\eqref{as:mu} allows us to relate stepsize expansions to improvements of the objective.

\begin{theorem} \label{t:deltak0}
  	Under Assumption \ref{ass:2}, if Inequality \eqref{as:mu} holds 
	then $\sum_{k \in \mathbb{N}_0} \me\left[\Delta_k^q\right] < \infty$.
\end{theorem}

\begin{proof}	
	Let $\ef = \sqrt[r(q)]{\eq}$, $\Phi_k = f(X_k) - f^* + \eta \Delta_k^q$, with $\eta = \frac{\theta}{\tqdel}$, and $\varepsilon = - \ef + \tqm\theta/\tqdel  > 0$
	where the inequality follows by \eqref{as:mu}.
	
	We will prove, for every $k \geq 0$, that
	\begin{equation} \label{eq:eaud}
		\me\left[\Phi_{k} - \Phi_{k + 1} \ | \ \F_{k - 1}\right] \geq \varepsilon \Delta_k^q \, .
	\end{equation}
	The thesis then follows as in \cite[Theorem 3]{dzahini2020expected}. 	
	
 Let $Z_k$ be the random variable such that $f(X_k) - f(X_k + \Delta_k G_k) = (\theta - Z_k) \Delta_{k}^q$, and let $J_k$ be the event that the step $k$ is successful.  We have 
	\begin{equation} \label{eq:longn}
		\begin{aligned}
			& \me\left[(\Phi_{k} - \Phi_{k + 1}) | \F_{k - 1}\right] = \me\left[(\Phi_{k} - \Phi_{k + 1})(\mo_{J_k} \;\; + (1 - \mo_{J_k})) | \F_{k - 1}\right] \\ 
			&	=  (f(X_k) - f(X_{k + 1}) + \eta(\Delta_{k}^q - \Delta_{k+1}^q))	\me\left[\mo_{J_k} | \F_{k - 1}\right] \\ 
			& \;\;\;\; + (f(X_k) - f(X_{k + 1}) + \eta(\Delta_{k}^q - \Delta_{k+1}^q))	\me\left[\ 1 - \mo_{J_k} | \F_{k - 1}\right] \\
			&	= (f(X_k) - f(X_k + \Delta_k G_k) + \eta(\Delta_{k}^q - \Delta_{k+1}^q)) \me\left[\mo_{J_k} | \F_{k - 1}\right] \\ 
			& \;\;\;\; + \eta(\Delta_{k}^q - \Delta_{k+1}^q)\me\left[ 1 - \mo_{J_k} | \F_{k - 1}\right] \\
			&	\geq  		(((\theta - Z_k) -\eta\tqp) \me\left[\mo_{J_k} | \F_{k - 1}\right] +  \eta \tqm \me\left[ 1 - \mo_{J_k} | \F_{k - 1}\right]) \Delta_{k}^q,
		\end{aligned}
	\end{equation}
	where we used $X_k = X_{k + 1}$ for unsuccessful steps in the second equality, and $\Delta_{k + 1} = \bar{\tau} \Delta_k \leq (1 + \tau) \Delta_k$ for successful steps in the inequality.
	In turn,
	\begin{equation} \label{eq:long2}
		\begin{aligned}
			& (((\theta - Z_k) -\eta\tqp) \me\left[\mo_{J_k} | \F_{k - 1}\right] +  \eta \tqm \me\left[ 1 - \mo_{J_k} | \F_{k - 1}\right]) \Delta_{k}^q \\
			&	= ((\theta -Z_k -\eta\tqdel)\me\left[\mo_{J_k} | \F_{k - 1}\right]  +  \eta\tqm)\Delta_{k}^q\\ 
			&	=  -Z_k \Delta_{k}^q \me\left[\mo_{J_k} | \F_{k - 1}\right] + \eta \tqm \Delta_{k}^q\, ,	
		\end{aligned}	
	\end{equation}
	where we used
	$\me[ 1 - \mo_{J_k} | \F_{k - 1}] = 1 - \me[\mo_{J_k} | \F_{k - 1}]$ in the first equality, and $\theta = \eta \tqdel$ in the second one.
	By combining \eqref{eq:longn} and \eqref{eq:long2} we can therefore conclude
	\begin{equation} \label{eq:long}
		\me\left[(\Phi_{k} - \Phi_{k + 1}) | \F_{k - 1}\right] \geq  -Z_k \Delta_{k}^q \me\left[\mo_{J_k} | \F_{k - 1}\right] + \eta\tqm \Delta_{k}^q \, .
	\end{equation}		 	
	
	Notice that if the step is successful then $f_k - f_k^g \geq \theta \delta_{k}^q$, which implies
	\begin{equation*}
	f_k - f_k^g- (f(x_k) - f(x_k + \delta_k g_k)) \geq  \theta \delta_{k}^q - (\theta - Z_k(\omega)) \delta_{k}^q = Z_k(\omega) \delta_{k}^q \, .
	\end{equation*}
	In particular $J_k \subset \{|\FMK - \FMG- (f(X_k) - f(X_k + \Delta_k G_k))| \geq Z_k \Delta_{k}^q \}$
	and we can write, for $Z_k^+ = Z_k \mo_{Z_k > 0}$, 
	\begin{equation} \label{eq:pj}
		\begin{aligned}
		& \me\left[\mo_{J_k} | \F_{k - 1}\right] = \me\left[\mo_{J_k} \mo_{\{Z_k > 0\}} +\mo_{J_k} \mo_{\{ Z_k \leq 0 \}}   | \F_{k - 1}\right] \\
		& =  \me\left[\mo_{J_k \cap \{Z_k > 0\}} | \F_{k - 1}\right] + \mo_{\{ Z_k \leq 0 \}}\me\left[\mo_{J_k}  | \F_{k - 1}\right]  \\ 
		&	\leq \Pb\left(|\FMK - \FMG- (f(X_k) - f(X_k + \Delta_k G_k))| \geq Z_k^+ \Delta_{k}^q | \F_{k - 1}\right) + \mo_{\{ Z_k \leq 0 \}}\me\left[\mo_{J_k}  | \F_{k - 1}\right] \, ,
		\end{aligned}
	\end{equation}
where we used the measurability of $Z_k$ w.r.t. $\F_{k - 1}$ in the second equality. 
	We now have
		\begin{equation} \label{eq:deltaineq}
	\begin{aligned}
		& -\rho_k \me\left[\mo_{J_k} | \F_{k - 1}\right] \geq -\rho_k^+\me\left[\mo_{J_k} | \F_{k - 1}\right] \\
		& \geq -\rho_k^+(\Pb\left(|\FMK - \FMG- (f(X_k) - f(X_k + \Delta_k G_k))| \geq \rho_k^+ \Delta_{k}^q | \F_{k - 1}\right) + \mo_{\{ \rho_k \leq 0 \}}\me\left[\mo_{J_k}  | \F_{k - 1}\right])  \\ 
		& = -\rho_k^+\Pb\left(|\FMK - \FMG- (f(X_k) - f(X_k + \Delta_k G_k))| \geq \rho_k^+ \Delta_{k}^q | \F_{k - 1}\right) \\
		& \geq -\rho_k^+ \min \left( 1, \varepsilon(\rho_k^+) \right) =  -\rho_k^+ \min \left(1, \varepsilon(\rho_k^+) \right) \geq  -\rho_k^+ \min \left(1, \varepsilon_1(\rho_k^+) \right) \geq -\ef \, ,	
	\end{aligned}
\end{equation}
where we applied \eqref{eq:pj} in the first inequality, the second inequality is a direct consequence of Lemma \ref{lem:condalpha} for $\balp = Z_k^+$, and $\varepsilon_1(t) =  + \infty \mo_{\{0\}} + \varepsilon_q/t \cdot \mo_{(0, +\infty)}$.
	Hence,
	\begin{equation} \label{eq:c2}
		\begin{aligned}
			-Z_k \Delta_{k}^q \me\left[\mo_{J_k} | \F_{k - 1}\right] + \eta\tqm \Delta_{k}^q  
			\geq  (- \ef + \eta\tqm)  \Delta_{k}^q = \varepsilon \Delta_{k}^q	\, ,	
		\end{aligned}
	\end{equation}
	where we used \eqref{eq:deltaineq} in the inequality. 
	
	Claim \eqref{eq:eaud} can finally be obtained by concatenating \eqref{eq:long} and \eqref{eq:c2}.
\end{proof}
	
	The next lemma will be useful for the
	proof of the optimality result of Theorem~\ref{t:cs} which is based on the Clarke generalized directional derivative. We notice that Assumption~\ref{ass:2} plays a key role in this result,
	allowing us to upper bound the error of the reduction estimate by a quantity that depends on the stepsize $\Delta_k$.

	\begin{lemma} \label{l:liminf}
	Let $K$ be the random set of indices of unsuccessful iterations.  Then under Assumption \ref{ass:2} and \eqref{as:mu}, a.s.~in $\OM$
		\begin{equation} \label{eq:liminfl}
			\liminf_{k \in K, \, k \rightarrow \infty} \frac{f(X_k + \Delta_k G_k) - f(X_k)}{\Delta_k} \geq 0 \, .
		\end{equation} 
	\end{lemma}
	
	\begin{proof}
	Clearly it suffices to show that, for any given $m \in \mathbb{N}$ and a.s.,
	\begin{equation} \label{eq:liminf}
		\liminf_{k \in K, \, k \rightarrow \infty} \frac{f(X_k + \Delta_k G_k) - f(X_k)}{\Delta_k} \geq -\frac{1}{m} \, .
	\end{equation}
	To start with, by applying Lemma \ref{lem:condalpha} with $\balp = \frac{\Delta_k^{1 - q}}{m}$ we have
	\begin{equation*}
		\Pb\left(|\FMK - \FMG- (f(X_k) - f(X_k + \Delta_k G_k))| \geq \frac{\Delta_{k}}{m} \ | \ \F_{k - 1} \right) \leq  m^{r(q)} \Delta_{k}^q \eq \, ,
	\end{equation*}
	and therefore taking expectations on both sides
	\begin{equation*}
		\Pb\left(|\FMK - \FMG- (f(X_k) - f(X_k + \Delta_k G_k))| \geq \frac{\Delta_{k}}{m} \right) \leq  m^{r(q)} \me\left[\Delta_{k}^q\right] \eq \, .
	\end{equation*}
	
	We can now deduce
	\begin{equation*}
		\sum_{k \in \mathbb{N}_0} \Pb\left(|\FMK - \FMG- (f(X_k) - f(X_k + \Delta_k G_k))| \geq \frac{\Delta_{k}}{m} \right) \leq \sum_{k \in \mathbb{N}_0}  m^{r(q)} \me\left[ \Delta_{k}^q\right] \eq  < \infty \, ,
	\end{equation*}
	where we applied Theorem \ref{t:deltak0} in the last inequality.
	In particular, by the Borel--Cantelli's first lemma
	\begin{equation*} 
		\Pb\left(\left\{|\FMK - f_k^g - (f(X_k) - f(X_k + \Delta_k G_k))| \geq \frac{\Delta_k}{m}\right\}\ \tx{i.o.}\right) = 0 \, ,
	\end{equation*}
	where ``i.o.'' stands for {\em infinitely often.} 
	Hence,
	we have a.s.
	\begin{equation} \label{eq:lenough}
		|F_k - F_k^g - (f(X_k) - f(X_k + \Delta_k G_k))| \leq \frac{\Delta_{k}}{m} \quad \mbox{for $k$ large enough}.
	\end{equation}
	From this we can infer that a.s., for every $k \in K$ large enough
	\begin{equation} \label{eq:limprel}
		\begin{aligned}
			& \frac{f(X_k + \Delta_k G_k) - f(X_k)}{\Delta_k} \geq \frac{\FMG- \FMK - |\FMK - \FMG- (f(X_k) - f(X_k + \Delta_k G_k))|}{\Delta_k}\\
			& \geq - \theta \Delta_k -\frac{1}{m} \,,
		\end{aligned}
	\end{equation}
	where we used \eqref{eq:lenough} combined with the unsuccessful step condition of Algorithm \ref{alg:GS} in the second inequality. Finally, \eqref{eq:liminf} follows passing to the liminf for $k \rightarrow \infty$ in \eqref{eq:limprel}. 
\end{proof}

	We now report the main convergence result for our stochastic direct-search scheme. \damiano{The result requires the existence of accumulation points for the sequence $\{x_k\}$, which can be obtained assuming that the iterates generated by the algorithm lie in a compact set as in \cite[Assumption 1]{audet2021stochastic}.}
	
	\begin{theorem} \label{t:cs}
		Assume that $f$ is Lipschitz continuous with constant $L_f^*$ around any limit point of the sequence of iterates~$\{X_k\}$.
		Let $K$ be the random set of indices of unsuccessful iterations.  Let Assumptions \ref{ass:2} and~\eqref{as:mu} hold. Then, the following property holds a.s.~in $\OM$: if $L \subset K$ is a random set such that the sequence $\{G_k\}_{k \in L}$ is dense in the unit sphere and $\lim_{k \in L, \, k \rightarrow \infty} X_k = X^*$, then the point $X^*$ is Clarke stationary, \damiano{i.e., $f^\circ(X^*, d) \geq 0$ for every $d \in \R^n$.}
	\end{theorem}
\begin{proof}
	\damiano{We refer to $\Ve$ as the event with probability one that \eqref{eq:liminfl} holds, and assume $\omega \in \Ve$ in the rest of the proof, with $L(\omega), K(\omega)$ satisfying the assumption described in the statement}. Let $d$ be a direction in the unit sphere, and let $\damiano{S(\omega) \subset L(\omega)}$ be such that $\lim_{k \in S(\omega), \, k\rightarrow \infty} G_k(\omega) = d \, .$ By definition of Clarke stationarity, since
	\begin{equation*}
		f^\circ(X^*(\omega), d) \geq  \limsup_{k \in S(\omega), \, k \rightarrow \infty} \frac{f(X_k(\omega) + \Delta_k(\omega) d) - f(X_k(\omega))}{\Delta_k(\omega)} \, ,
	\end{equation*}
	we just need to prove that on $\Ve$, and therefore a.s.,
	\begin{equation*}
		\limsup_{k \in S(\omega), \, k \rightarrow \infty} \frac{f(X_k(\omega) + \Delta_k(\omega) d) - f(X_k(\omega))}{\Delta_k(\omega)} \geq 0 \, .
	\end{equation*} 
	
	For  $\omega \in \Ve$ we can write 
	\begin{equation} \label{eq:easy}
		\begin{aligned}
			& \limsup_{k \in S(\omega), \, k \rightarrow \infty} \frac{f(X_k(\omega) + \Delta_k(\omega) G_k(\omega)) - f(X_k(\omega))}{\Delta_k(\omega)} \\ 
			& \geq \liminf_{k \in K(\omega), \, k \rightarrow \infty} \frac{f(X_k(\omega) + \Delta_k(\omega) G_k(\omega)) - f(X_k(\omega))}{\Delta_k(\omega)} \geq 0 \, ,	
		\end{aligned}
	\end{equation}
	where the last inequality follows by \eqref{eq:liminfl}.
	
	Now using the Lipschitz property of $f$ we can write, for $k \in S(\omega)$ large enough,
	\begin{equation*}
		\begin{aligned}
			& \frac{f(X_k(\omega) + \Delta_k(\omega) d) - f(X_k(\omega))}{\Delta_k(\omega)}  \\
			& = \frac{f(X_k(\omega) + \Delta_k(\omega) G_k(\omega)) - f(X_k(\omega))}{\Delta_k(\omega)} +  \frac{f(X_k(\omega) + \Delta_k(\omega) d) - f(X_k(\omega) + \Delta_k(\omega) G_k(\omega))}{\Delta_k(\omega)}  \\
			& \geq \frac{f(X_k(\omega) + \Delta_k(\omega) G_k(\omega)) - f(X_k(\omega))}{\Delta_k(\omega)} - L_f^* \n{G_k(\omega) - d}.
		\end{aligned}
	\end{equation*}
	
	Passing to the limsup for $k \in S(\omega)$ we get
	\begin{equation*}
		\begin{aligned}
			& \limsup_{k \in S(\omega), \, k \rightarrow \infty} \frac{f(X_k(\omega) + \Delta_k(\omega) d) - f(X_k(\omega))}{\Delta_k(\omega)} \\
			& \geq \limsup_{k \in S(\omega), \, k \rightarrow \infty} \frac{f(X_k(\omega) + \Delta_k(\omega) G_k(\omega)) - f(X_k(\omega))}{\Delta_k(\omega)} \geq 0 \, ,
		\end{aligned}
	\end{equation*}
	for every $\omega \in \Ve$, where we used $\n{G_k(\omega) - d} \rightarrow 0$ by construction in the first inequality and~\eqref{eq:easy} in the second.
\end{proof}

	\section{A simple trust-region method for stochastic non-smooth functions}\label{trm}

After having analyzed a simple stochastic direct-search method, we focus on a stochastic version of the Basic DFO-TRNS presented in~\cite{liuzzi2019trust}, and analyze its convergence properties under tail bound probabilistic conditions like the ones used in Section~\ref{direct}. Some minor changes in notation are convenient and will be introduced with a clear reference to the corresponding elements of Algorithm~\ref{alg:GS}.
	
	\subsection{The stochastic trust-region scheme} \label{s:trscheme}

	As already mentioned, the simple trust-region algorithm that we reported here is a  minor modification of  the Basic DFO-TRNS algorithm proposed in \cite{liuzzi2019trust}. Indeed, there are two differences between the Basic DFO-TRNS algorithm and its stochastic counterpart.
	
	The first difference is in the  updating rule related to the trust-region radius. In the modification presented in this work, $\tau \in (0,1)$ is chosen, with $1-\tau$ corresponding contraction factor and $\bar{\tau} \in [1,1+\tau]$ as expansion factor. 
	
	The second, more relevant difference is the fact that the linear term $g_k$ is not constrained to the unit sphere as is the case in DFO-TRNS. This makes more sense when modeling cases where~$g_k$ resembles an approximation of the gradient.
		
	The detailed scheme is reported in Algorithm~\ref{alg:DFO-TRNS}. At every iteration~$k$, a symmetric matrix~$B_k$ is built from interpolation or regression on a sample set of points. The linear term~$g_k$ needs to randomly cover the unit sphere when normalized.
 By using these quantities, a quadratic model of the objective function around $x_k$ is built. The step $s_k$ is obtained by solving the trust-region subproblem, i.e., by minimizing the quadratic model within the spherical trust-region constraint. Once the current step has been computed, the algorithm generates an estimate of the true objective function~$f$ at the trial point $x_k+s_k$ and recomputes a new estimate at~$x_k$, after which the acceptance ratio~$\bar{\rho}_k$ is computed.
	Note that, as in \cite{liuzzi2019trust}, the non-standard acceptance ratio is motivated by convergence requirements.  In this scheme, realizations related to the estimates of the function values $f(x_k)$ at the current iterate and $f(x_k+s_k)$ at the potential next iterate are indicated with $\f$ and $\fs$, thus replacing $f_k^g$ used in the direct-search scheme, as a shorthand for $F_k(\omega)$ and $F_k^s(\omega)$, respectively.
	
	{\linespread{1.3}
		
		\begin{algorithm}   
			\caption{\texttt{Stochastic DFO Trust-Region Algorithm}}
			\label{alg:DFO-TRNS} 
			\begin{algorithmic} 
				\par\vspace*{0.1cm}
				\item$\,\,\,1$ \textbf{Initialization.} Select $x_0\in \mathbb{R}^n$, $\theta > 0$, $\tau\in(0,1)$, $\bar{\tau} \in [1, 1 + \tau]$, $\delta_0>0$, $q > 1$.
				\item$\,\,\,2$ {\bf For} $k=0, 1\ldots$
            
				\item$\,\,\,3$ \quad Select a direction $g_k \neq 0$ and build a symmetric matrix $B_k$.
				\item$\,\,\,4$ \quad Compute
				\vspace*{-0.5cm}
				\begin{equation*}
					s_k  \in  \argmin_{\|s\|\leq \delta_k} g_k^\top s + \frac{1}{2}s^\top B_k s.
				\end{equation*}
				\item$\,\,\,5$ \quad Compute estimates $f_k$, $f_k^s$ for $f$ at $x_k$, $x_k + s_k$, respectively, and let 
				\[
				\bar{\rho}_k = \frac{f_k - f_k^s}{\theta \| s_k \|^{q}}.
				\]
				\item$\,\,\,6$ \quad \textbf{If} $\bar{\rho}_k \geq 1$ \textbf{Then} set {\tt SUCCESS} $=$ {\tt true}, $x_{k+1}= x_k+s_k$,
				 $\delta_{k + 1} = \bar{\tau} \delta_k$.
				\item$\,\,\,7$ \quad \textbf{Else} set {\tt SUCCESS} $=$ {\tt false}, $x_{k+1}= x_k$, $\delta_{k+1}= (1-\tau)\delta_k$.
				\item$\,\,\,8$ \quad \textbf{End If}
				\item$\,\,\,9 $ {\bf End For}
				\par\vspace*{0.1cm}
			\end{algorithmic}
		\end{algorithm}
	}
	
	For convergence purposes, we require the  Hessian model to satisfy the assumption below.

	\begin{assumption} \label{ass:modelHessian}
	There exist $\rho \in (0,1]$ such that, for every $k\in \mathbb{N}_0$, 
		$\n{B_k} \leq	\frac{1}{\rho} \frac{\n{G_k}}{\Delta_k}$.
	\end{assumption}
 
When $\n{G_k} = 1$, the above assumption is essentially saying that $B_k$ can be unbounded as long as it does not go to infinity faster than $1/\Delta_k$, that is a weaker version of \cite[Assumption 2.1]{liuzzi2019trust}.

We now show, under Assumption~\ref{ass:modelHessian}, that the norm $\n{S_k}$ of every trust-region subproblem solution is equal to $\Delta_k$, up to a constant. This will allow us to deduce convergence to~0 of the trust-region radius from convergence to~0 of the solution norm.

\begin{lemma} \label{l:modelH}
	Under Assumption \ref{ass:modelHessian}, $\n{S_k} \geq  \rho \Delta_k$.
\end{lemma}

\begin{proof}
	The thesis is clear if $S_k$ is on the boundary of the trust region, which includes the case $B_k = 0$ since $G_k \neq 0$ by assumption. Otherwise, if $S_k$ is in the interior we must have $B_kS_k = - G_k$, and therefore $ \n{B_k}\n{S_k} \geq \n{G_k} \geq \rho \Delta_k \n{B_k}$,
where we used Assumption \ref{ass:modelHessian} in the second inequality, and the proof is completed.
\end{proof} 

	\subsection{Convergence analysis under the tail bound probabilistic condition} \label{s:catr}
	
In order to analyze the method introduced above, we adapt Assumption \ref{ass:2}, replacing $G_k$ with $\hat{S}_k$,  \damiano{using the $\hat{\cdot}$ notation introduced at the beginning of Section \ref{condition}}, and $\Delta_k$ with $\n{S_k}$.   Now~$\Delta_k$ stands for the trust-region radius.
	Hence, we obtain the following tail bound condition.
 
\begin{assumption} \label{ass:2bis}
For some $\eq > 0$ independent of~$k$:	
\begin{equation*} 
	\Pb\left(|\FMK - \FMG- (f(X_k) - f(X_k + S_k))| \geq \alpha \n{S_k}^{q} \ | \F_{k - 1} \right) \leq \frac{\eq}{\alpha^{q/(q - 1)}}
\end{equation*}
a.s.~for every $\alpha > 0$.
\end{assumption}
	Importantly, if $B_k$ is random the definition of $\F_{k-1}$ must be modified as the $\sigma$-algebra of events up to the generation of $B_k$ and $g_k$. \\ 
	The next theorem implies convergence of the series of trust-region radii elevated to the $q$ almost surely. This obviously implies that the trust-region radius converges to zero almost surely. 
 
	\begin{theorem} \label{t:deltak0bis}
		Under Assumptions \ref{ass:modelHessian} and \ref{ass:2bis}, if
		\begin{equation} \label{as:mubis}
			\theta > \frac{(\rho^q\tqm + \tqp)\sqrt[r(q)]{\eq}}{\rho^q\tqm} \, ,
		\end{equation}
		then $\sum_{k \in \N_0} \me\left[\trr_k^q\right] < \infty$.
	\end{theorem}
 
	\begin{proof}	
			Let $\ef = \sqrt[r(q)]{\eq}$, $\Phi_k = f(X_k) - f^* + \eta \n{S_k}^q$, with $\eta = \frac{\theta \rho^q}{\tqp + \rho^q\tqm}$. Let also $\varepsilon = - \ef + \frac{\rho^q\theta}{\tqp + \rho^q\tqm}  > 0,$  where the inequality follows by \eqref{as:mubis}.
			We will prove, for every $k \geq 0$, that
			\begin{equation} \label{eq:eaud2}
				\me\left[\Phi_{k} - \Phi_{k + 1} \ | \ \F_{k - 1}\right] \geq \varepsilon \n{S_k}^q \, .
			\end{equation}
			Then for the same reasons stated in the proof of Theorem \ref{t:deltak0}, with $\n{S_k}$ instead of $\Delta_k$, we get
			\begin{equation*}
				\sum_{k \in \mathbb{N}_0} \me\left[\n{S_k}^q\right] < + \infty \, ,
			\end{equation*}
			and therefore
			\begin{equation*}
				\sum_{k \in \mathbb{N}_0} \me\left[\Delta_k^q\right] \leq \frac{1}{\rho^q}	\sum_{k \in \mathbb{N}_0}\me\left[\n{S_k}^q\right] <  + \infty \, ,
			\end{equation*}
			where we used Assumption \ref{ass:modelHessian} in the inequality.
			Let $Z_k$ be the random variable such that $f(X_k) - f(X_k + S_k) = (\theta - Z_k)  \n{S_k}^q$, and let $J_k$ be the event that the step $k$ is successful.  We have 
			\begin{equation} \label{eq:longn2}
				\begin{aligned}
					& \me\left[(\Phi_{k} - \Phi_{k + 1}) | \F_{k - 1}\right]  \\
					&	\geq (f(X_k) - f(X_k + S_k) - \eta\tqp\Delta_k^q) \me\left[\mo_{J_k} | \F_{k - 1}\right] \\ 
					& \;\;\;\; + \eta(\Delta_{k}^q - \Delta_{k+1}^q)\me\left[ 1 - \mo_{J_k} | \F_{k - 1}\right] \\
					& \geq (\theta - Z_k)\me\left[\mo_{J_k} | \F_{k - 1}\right]\n{S_k}^q - \eta\tqp\Delta_k^q \me\left[\mo_{J_k} | \F_{k - 1}\right] + \eta \tqm \Delta_{k}^q \me\left[ 1 - \mo_{J_k} | \F_{k - 1}\right]  \\ 
					& \geq (\theta - Z_k)\me\left[\mo_{J_k} | \F_{k - 1}\right]\n{S_k}^q - \eta\frac{\tqp}{\rho^q}\n{S_k}^q \me\left[\mo_{J_k} | \F_{k - 1}\right] + \eta \tqm \n{S_k}^q \me\left[ 1 - \mo_{J_k} | \F_{k - 1}\right]  \\ 
					& = \left((\theta - Z_k)  - \eta\frac{\tqp}{\rho^q} \me\left[\mo_{J_k} | \F_{k - 1}\right] + \eta \tqm \me\left[ 1 - \mo_{J_k} | \F_{k - 1}\right] \right) \n{S_k}^q
				\end{aligned}
			\end{equation}
			where the first inequality follows as in \eqref{eq:longn}, the second inequality by definition of $Z_k$, and the third inequality we use \eqref{ass:modelHessian} on the second summand and $\n{S_k} \leq \Delta_k$ in the third summand. 
			In turn,
			\begin{equation} \label{eq:long22}
				\begin{aligned}
					& \left((\theta - Z_k)  - \eta\frac{\tqp}{\rho^q} \me\left[\mo_{J_k} | \F_{k - 1}\right] + \eta \tqm \me\left[ 1 - \mo_{J_k} | \F_{k - 1}\right] \right) \n{S_k}^q \\
					&	= \left(\theta -Z_k - \eta\left(\frac{\tqp}{\rho^q} + \tqm\right) \me\left[\mo_{J_k} | \F_{k - 1}\right]  +  \eta\tqm \right)\n{S_k}^q\\ 
					&	=  -Z_k \n{S_k}^q \me\left[\mo_{J_k} | \F_{k - 1}\right] + \eta \tqm \n{S_k}^q\, ,	
				\end{aligned}	
			\end{equation}
			where we used
			$\me[ 1 - \mo_{J_k} | \F_{k - 1}] = 1 - \me[\mo_{J_k} | \F_{k - 1}]$ in the first equality, and  $\eta = \frac{\theta \rho^q}{\tqp + \rho^q\tqm}$ in the second one.
			By combining \eqref{eq:longn2} and \eqref{eq:long22} we therefore get
			\begin{equation*}
				\me\left[(\Phi_{k} - \Phi_{k + 1}) | \F_{k - 1}\right] \geq  -Z_k\n{S_k}^q\me\left[\mo_{J_k} | \F_{k - 1}\right] + \eta\tqm \Delta_{k}^q \, .
			\end{equation*}		 				
			The conclusion now follows as in the proof of Theorem \ref{t:deltak0}, replacing $\Delta_k$ with $\n{S_k}$.
 	\end{proof}

As for the analysis of our direct-search scheme in Section \ref{direct}, we now state a lemma that  will be useful for the
proof of the optimality result based on the Clarke generalized derivative.

	\begin{lemma} \label{l:liminfbis} 
		Let $K$ be the random set of indices of unsuccessful iterations.  Then under Assumptions \ref{ass:modelHessian}, \ref{ass:2bis}, and \eqref{as:mubis}, a.s.
		\begin{equation*}
			\liminf_{k \in K, \, k \rightarrow \infty} \frac{f(X_k + S_k) - f(X_k)}{\|S_k\|} \geq 0 \, .
		\end{equation*} 
	\end{lemma}
	
	\begin{proof}
	Follows analogously to Lemma \ref{l:liminf}.
	\end{proof}

	We now state a convergence result extending Theorem \ref{t:cs} to our trust-region method.

	\begin{theorem} \label{t:csbis}
	Assume that $f$ is Lipschitz continuous with constant $L_f^*$ around any accumulation point of the sequence of iterates~$\{X_k\}$.
		Let $K$ be the random set of indices of unsuccessful iterations.  Let Assumptions \ref{ass:modelHessian}, \ref{ass:2bis}, and \eqref{as:mubis} hold. Then, the following property holds a.s.~in $\OM$: if $L \subset K$ is a random set such that $\{\hat{S}_k\}_{k \in L}$ is dense in the unit sphere and $\lim_{k \in L, \, k \rightarrow \infty} X_k = X^*$, then the point $X^*$ is Clarke stationary, \damiano{i.e., $f^\circ(X^*, d) \geq 0$ for every $d \in \R^n$.}
	\end{theorem}
	
	\begin{proof}
		The proof follows the lines of Theorem~\ref{t:cs}'s proof, replacing $\Delta_k$ and $G_k$ by $\|S_k\|$
		and $\hat{S}_k$, respectively.
	\end{proof}	
 
We now introduce a stronger version of Assumption~\ref{ass:modelHessian}, and show that under this stronger assumption the trust-region scheme becomes at the limit a search along a direction $G_k$ with stepsize~$\Delta_k$. 

\begin{assumption}\label{ass:modelHessian2}
For some positive sequence \damiano{of uniformly bounded random variables} $\{A_k\}$ such that $A_k \rightarrow 0$ a.s., it holds a.s.~$\n{B_k} \leq A_k \n{G_k}/\Delta_k$.
\end{assumption}
Trivially, Assumption \ref{ass:modelHessian2} implies Assumption \ref{ass:modelHessian}, with $\rho = \frac{1}{\max(\{\n{A_k}_{\infty}\})} $.

	\begin{proposition}\label{property:Cauchy}
		Let Assumptions \ref{ass:2bis}, \ref{ass:modelHessian2}, and \eqref{as:mubis} hold. Then a.s.~$\lim_{k \rightarrow \infty} \hat{G}_k + \hat{S}_k = 0$.
	\end{proposition}
 
\begin{proof}
First, notice that $\n{\hat{G}_k} = 1$, as well as $\n{\hat{S}_k} = 1$ since $G_k$ must be always different from~$0$ and therefore $S_k$ as well. Now define $F_k^m$ as the local model $F_k^m(s) = G_k^\top s + \frac{1}{2}s^\top B_k s$, and let $\Gamma_k = \hat{G}_k^\top \hat{S}_k$ be the cosine of the angle between $\hat{G}_k$ and $\hat{S}_k$. We need to prove $\Gamma_k \rightarrow -1$ almost surely.

We have on the one hand
\begin{equation} \label{eq:Fkm}
	\begin{aligned}
			& F_k^m(S_k) = S_k^\top G_k + \frac{1}{2} S_k^\top B_k S_k = \Gamma_k\n{S_k}\n{G_k}  + \frac{1}{2} S_k^\top B_k S_k \\ 
			& \geq \min(0, \Gamma_k) \Delta_{k} \n{G_k} - \frac{1}{2} \n{B_k} \Delta_k^2 \, ,
	\end{aligned}
\end{equation}
where we used $\n{S_k} \leq \Delta_k$ in the inequality. On the other hand
\begin{equation} \label{eq:Fkm2}
	F_k^m(-\Delta_{k}\hat{G}_k) = -\Delta_k\n{G_k} + \frac{\Delta_{k}^2}{2}\hat{G}_k^\top B_k \hat{G}_k \leq -\Delta_k\n{G_k} + \frac{1}{2}\Delta_{k}^2\n{B_k} \, .
\end{equation}
Putting \eqref{eq:Fkm} and \eqref{eq:Fkm2} together we obtain
\begin{equation} \label{eq:intermediate}
	-\Delta_k\n{G_k} + \frac{1}{2}\Delta_{k}^2\n{B_k} \geq F_k^m(-\Delta_{k}\hat{G}_k) \geq F_k^m(S_k) \geq \min(0, \Gamma_k) \Delta_{k} \n{G_k} - \frac{1}{2} \n{B_k} \Delta_k^2 \, ,
\end{equation}
where in the second inequality we used that $S_k$ is a solution of the trust-region subproblem. Then rearranging \eqref{eq:intermediate} and dividing by $\Delta_k\n{G_k}$ we get
\begin{equation} \label{eq:final}
	(1 + \min(0, \gamma_{k})) \leq  \frac{\n{B_k}\Delta_{k}}{\n{G_k}} \, .
\end{equation}
Since the right-hand side of \eqref{eq:final} converges to $0$ a.s.~\damiano{thanks to Assumption \ref{ass:modelHessian2}}, we get $1 + \min(0, \gamma_{k}) \rightarrow 0$ a.s., and we can conclude $\Gamma_k\rightarrow -1$ a.s.~as desired.
\end{proof}
Under the conditions of Proposition \ref{property:Cauchy}, we just need to ensure that $\hat{G}_k$ is dense in the unit sphere on subsequences to obtain convergence to Clarke stationary points, as expressed in the following corollary.

\begin{corollary}\label{cor:convtr_classic}
	Assume that $f$ is Lipschitz continuous with constant $L_f^*$ around any accumulation point of the sequence of iterates~$\{X_k\}$.
Let $K$ be the random set of indices of unsuccessful iterations.  Let Assumptions \ref{ass:2bis}, \ref{ass:modelHessian2} and \eqref{as:mubis}  hold. Then, the following property holds a.s.~in $\OM$: if $L \subset K$ is a random set such that the sequence $\{\hat{G}_k\}_{k \in L}$ is dense in the unit sphere and $\lim_{k \in L, \, k \rightarrow \infty} X_k = X^*$ then the point $X^*$ is Clarke stationary, \damiano{i.e., $f^\circ(X^*, d) \geq 0$ for all $d \in \mathbb{R}^n$.}
	\end{corollary}	
 
	\begin{proof}
		Thanks to Proposition \ref{property:Cauchy}, for almost every $\omega$ in $\Omega$, if the sequence $\{\hat{G}_k(\omega)\}_{k \in L}$  is dense in the unit sphere $\{\hat{S}_k(\omega)\}_{k \in L}$ also is, and we can therefore apply Theorem \ref{t:csbis}.
	\end{proof}
 
\section{Numerical results}\label{s:numres}
We report here some numerical results, first comparing the performance of Algorithm \ref{alg:GS} for different choices of $q$, and then comparing Algorithms \ref{alg:GS} and \ref{alg:DFO-TRNS} to StoMADS from \cite{audet2021stochastic}. \\
   To compare the performance of the algorithms, we use data and performance profiles as defined in \cite{more2009benchmarking}. Their definitions are briefly recalled here.	Given a set~$S$ of algorithms and a set~$P$ of problems, for $s\in S$ and $p \in P$, let $t_{p,s}$ be the number of function evaluations required by algorithm $s$ on problem~$p$ to satisfy the condition 
\begin{equation}\label{eq:stop}
	f(x_k) \leq f_L + \gamma_p(f(x_0) - f_L)\, , 
\end{equation}
where $\gamma_p > 0$ and $f_L$ is the best objective function value achieved by any solver on problem~$p$. Then, the performance and data profiles of solver $s$ are generated using
\begin{eqnarray*}
	\rho_s(\alpha) & = & \frac{1}{|P|}\left|\left\{p\in P: \frac{t_{p,s}}{\min\{t_{p,s'}:s'\in S\}}\leq\alpha\right\}\right|,\\
	d_s(\kappa) & = & \frac{1}{|P|}\left|\left\{p\in P: t_{p,s}\leq\kappa(n_p+1)\right\}\right|\, ,
\end{eqnarray*}
where $n_p$ is the dimension of problem $p$.
A budget of $10000(n_p+1)$ sample evaluations for both algorithms is used, and two different tolerances for \eqref{eq:stop}, that is $\gamma_p \in \{10^{-2},10^{-4}\}$. All the profiles are built with the true function values, while applying
the algorithms to the noisy functions. The set $P$ includes 96 well known instances of derivative-free unconstrained nonsmooth optimization problems. \damiano{The full problem list, with dimensions and references, is reported in an appendix (see Table~\ref{tab:1} in Section \ref{sec:bench})}. Each of the instances is used 10 times, so that the algorithms perform 10 runs on every instance, thus getting $|P| = 960$. 

	\subsection{Algorithm \ref{alg:GS} for different choices of $q$} \label{s:numresAP}
In this section, we compare two basic instances of Algorithm \ref{alg:GS}, obtained choosing uniformly at random the search direction in the unit sphere, for different choices of the sufficient decrease parameter and sampling strategies, corresponding to different values of $q$ and $r$ in the algorithmic scheme and in the assumptions. The main goal is to provide further evidence that choosing~$q$ smaller than~$2$ and using fewer samples per iteration as suggested by the theory can improve numerical performance. In particular, we will show that the claim remains true also in the case of correlated errors discussed in Section~\ref{s:si}. 
\begin{remark}
	It is of course not always the case that an improvement in number of samples per iteration leads to an improvement in the solution found with a fixed budget of samples, since using lower values of~$q$ might increase the iteration complexity. For instance, for smooth objectives with deterministic oracles, a complexity of $O(\epsilon^{-\frac{q}{q - 1}})$ was proved in \cite{vicente2013worst} for a scheme analogous to Algorithm \ref{alg:GS}, with  $q \in (1, 2]$. Then in this case, the lower number of samples per iteration for $q$ approaching 1 comes at the price of a potentially much higher iteration complexity. However, it is important to note that the complexity bounds from~\cite{vicente2013worst} heavily rely on the Lipschitz continuity of the gradient, so that this trade-off does not necessarily generalize to potentially non-smooth objectives.
\end{remark}

The basic version of Algorithm \ref{alg:GS} used here is referred to as SDS$q$ for $q \in \{2, 1.5\}$.  We are therefore comparing a standard choice \cite{audet2021stochastic,dzahini2020expected,dzahini2022constrained} to one allowing the use of a lower number of samples per iteration as proved in Theorems~\ref{th:a1} and~\ref{th:corrbound}. The noise on the objective was assumed to be $0$ in expectation and normally distributed with standard deviation $0.1$.  By Theorem~\ref{th:a1}, $O(\Delta_k^{-2q})$ samples are needed to satisfy the weak tail bound assumptions. Given that the Gaussian noise has finite $r$-th moment for every $r$, Theorem~\ref{th:a1} can be applied with $r = q/(q - 1)$. The number of samples needed per iteration is then $O(\Delta_{k}^{-4})$ and $O(\Delta_{k}^{-3})$ respectively for $q = 2$ and $q = 1.5$. Thus, we simulated the resulting noise after averaging $p_k$ independent samples  by adding to the objective $N(0, 1/ \sqrt{p_k})$ distributed random variables. The remaining parameters were tuned with a basic grid search to obtain good performances for both instances of Algorithm \ref{alg:GS} to $\tau = 0.001$, $\bar{\tau} = 1.001$, $\theta = 0.5$ and $\delta_0 = 2$.   

\begin{remark}
	It is not difficult to check that the bound \eqref{as:mu} translates to $\theta > c$ with $c = 4$ and $c \approx 9$ for $q = 2$ and $q = 1.5$ respectively. However, both algorithms show bad performance for $\theta$ greater than 1. The authors conjecture here that lower values of $\theta$ and therefore a more tolerant acceptance test might still lead to convergence in practice in most cases, with a lower number of samples needed to find a good solution due to the resulting more aggressive exploration. Finding weaker versions of \eqref{as:mu} that still guarantee convergence under reasonable assumptions remains of course an open problem to be studied more in depth in future works.   
\end{remark}

Data and performance profile in the general case of finite $r$-th moment and correlated errors are reported in Figures \ref{fig:overallAP} and \ref{fig:overall2AP} respectively. In the case of finite $r$-th moment, the number of samples was set equal to $p_k = \lceil 0.01 \delta_{k}^{-4} \rceil$ and $p_k = \lceil 0.01 \delta_{k}^{-3} \rceil$ for $q=2$ and $q = 1.5$ respectively, consistently with the bound proved in Theorem \ref{th:a1}. In the correlated error case, at the iteration $k$ we add noise with standard deviation $0.1\delta_k$ in the estimate of the difference and set the number of samples  to $\lceil 0.01 \delta_k^{-2} \rceil$ and $\lceil 0.01 \delta_k^{-1} \rceil$ for $q = 2$ and $q = 1.5$ respectively, consistently with the bound proved in Theorem \ref{th:corrbound}.  For both cases it can clearly be seen how choosing $q = 1.5$ rather than $q = 2$ leads to a better performance.

\begin{figure}[h]
	\centering
	\begin{subfigure}[b]{0.34\textwidth}
		\includegraphics[width=\linewidth]{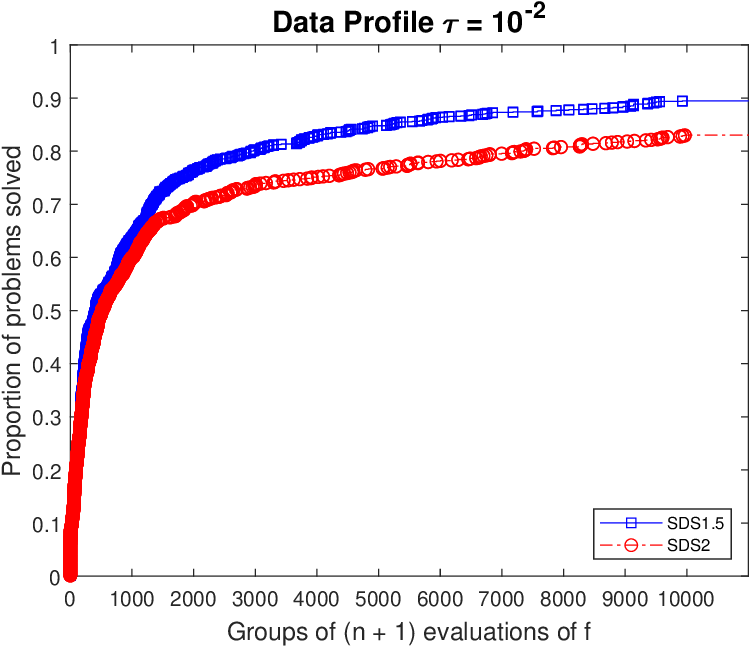}
	\end{subfigure}	
	\begin{subfigure}[b]{0.34\textwidth}
		\includegraphics[width=\linewidth]{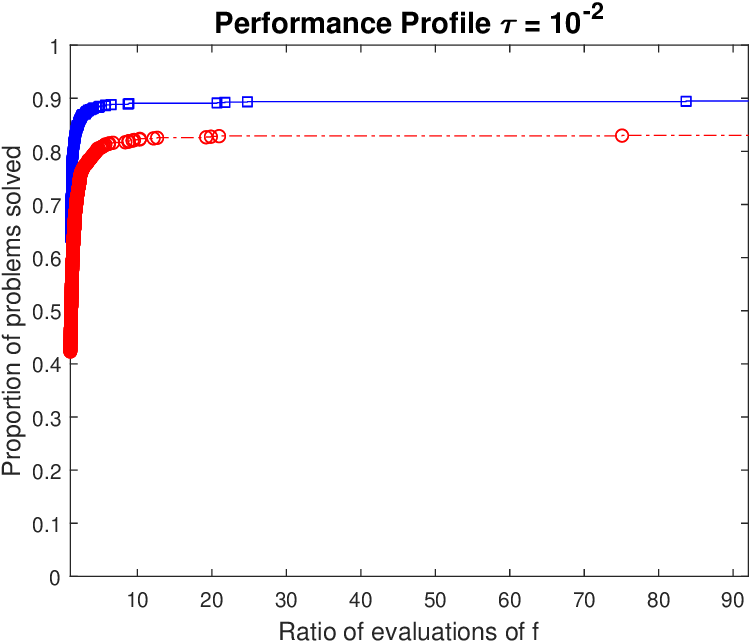}
	\end{subfigure}

	\begin{subfigure}[b]{0.34\textwidth}
		\includegraphics[width=\linewidth]{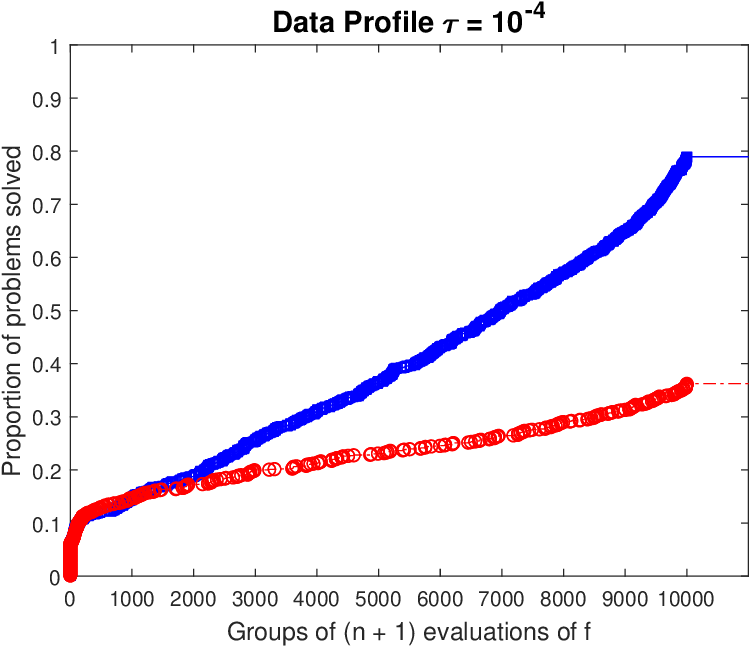}
	\end{subfigure}
	\begin{subfigure}[b]{0.34\textwidth}
		\includegraphics[width=\linewidth]{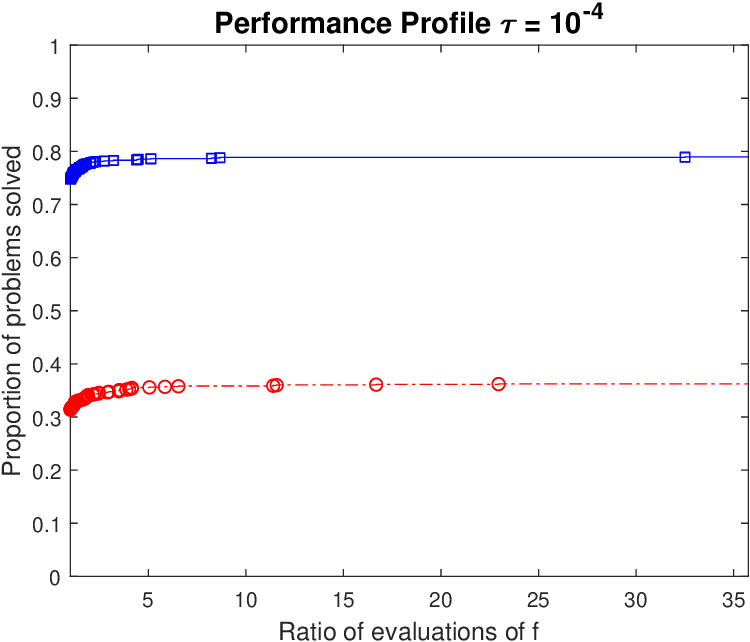}
	\end{subfigure}
	\caption{Data and performance profiles for Algorithm \ref{alg:GS} with $q \in \{2, 1.5\}$ and in the finite $r-$th moment setting, corresponding to SDS2 and SDS1.5, on the set of problems reported in Table \ref{tab:1}.}
	\label{fig:overallAP}
\end{figure}

\begin{figure}[h]
	\centering
	\begin{subfigure}[b]{0.34\textwidth}
		\includegraphics[width=\linewidth]{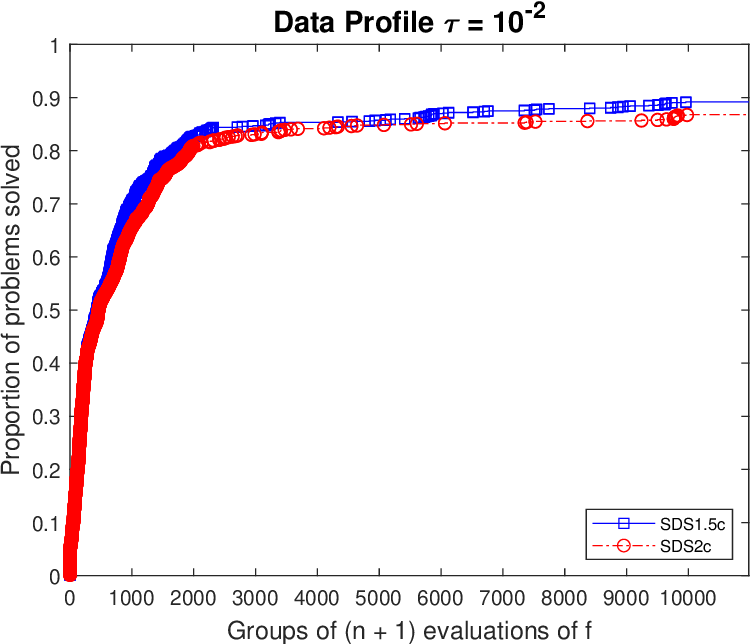}
	\end{subfigure}	
	\begin{subfigure}[b]{0.34\textwidth}
		\includegraphics[width=\linewidth]{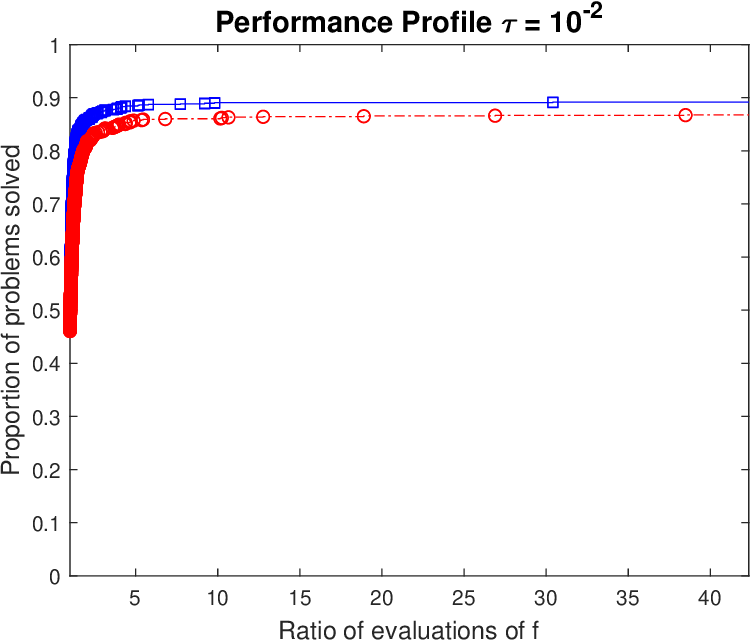}
	\end{subfigure}
	
	\begin{subfigure}[b]{0.34\textwidth}
		\includegraphics[width=\linewidth]{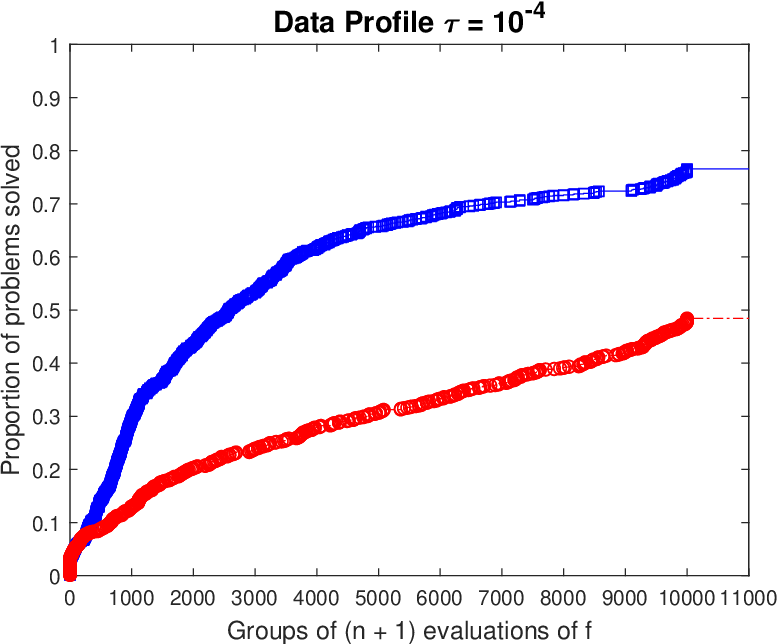}
	\end{subfigure}
	\begin{subfigure}[b]{0.34\textwidth}
		\includegraphics[width=\linewidth]{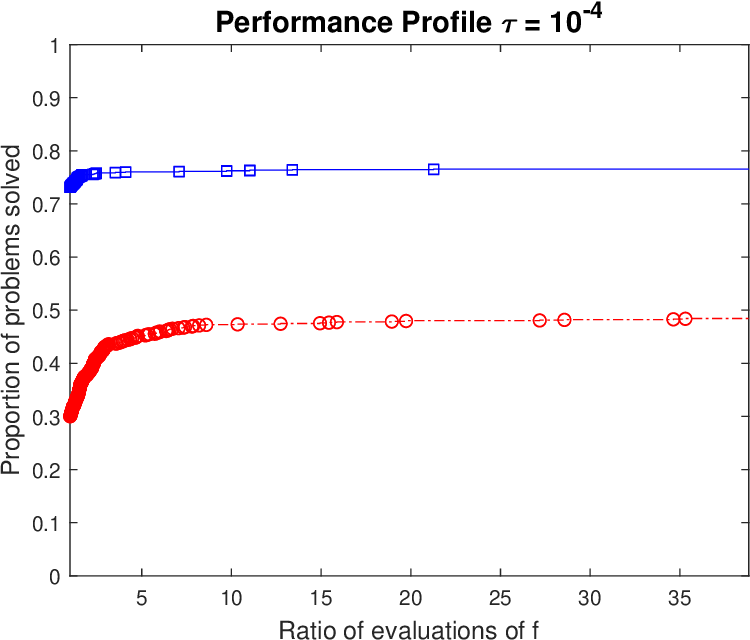}
	\end{subfigure}
	\caption{Data and performance profiles for Algorithm \ref{alg:GS} with $q \in \{2, 1.5\}$ and in the correlated error setting, corresponding to SDS2c and SDS1.5c, on the set of problems reported in Table~\ref{tab:1}.}
	\label{fig:overall2AP}
\end{figure}

\subsection{Comparison of Algorithms \ref{alg:GS} and \ref{alg:DFO-TRNS} with STOMADS}

\damiano{We describe in this section numerical results comparing a modified version of Algorithm~\ref{alg:GS}, Algorithm \ref{alg:DFO-TRNS}, and the StoMADS algorithm from \cite{audet2021stochastic}. The version of Algorithm~\ref{alg:GS} considered here is obtained alternating coordinate search directions with random directions after the stepsize falls below a certain threshold~$\bar{\Delta}$, like it was done in the deterministic case, e.g., in \cite{fasano2014linesearch,kungurtsev2022retraction}. The convergence result of Theorem~\ref{t:cs} extends to this variant in a straightforward way. In the tests, we used the threshold $\bar{\Delta} = 0.5$.  As for Algorithm \ref{alg:DFO-TRNS}, we adopt at all iterations the approach described in \cite[Section 5]{bandeira2012computation} to build the model at iteration~$0$, i.e., we build a minimum Frobenius norm model using the sample set $\{x_k\} \cup \{x_k \pm \delta_k e_i \}_{i \in [1:n]}$ at all iterations~$k$. Unlike in \cite[Algorithm 5.1]{bandeira2012computation} we do not add or subtract any point to the sample set, and rebuild from scratch the model at every iteration.}
  
\damiano{For both Algorithm \ref{alg:GS} and Algorithm \ref{alg:DFO-TRNS}, two choices of the parameter $q$ were tested, that is $q = 2$ and $q = 1.5$. The two instances of the modified version of Algorithm \ref{alg:GS} are referred to as SDS+$q$ for $q\in\{2,1.5\}$, and analogously the two instances of Algorithm \ref{alg:DFO-TRNS} are referred to as STR$q$ for $q \in \{2, 1.5\}$.  In accordance with this bound and the one in our Theorem \ref{th:a1}, the number of samples was set equal to $p_k = \lceil 0.01 \delta_{k}^{-4} \rceil$ and $p_k = \lceil 0.01 \delta_{k}^{-3} \rceil$ for $q=2$ and $q = 1.5$ respectively, like in the previous section. In the case of StoMADS, when using the default choice $q = 2$ for the frame size exponent in the acceptance criterion, the theory in \cite{audet2021stochastic} suggests the use of $O(\Delta_{k}^{-4})$ samples per iteration.}
      
   \damiano{By taking a look at the profiles, it can be easily seen that SDS+1.5 and STR1.5 outperform the other methods for $\gamma_p = 10^{-2}$ and $\gamma_p = 10^{-4}$ respectively. This suggests that the algorithms analyzed in this work can outperform StoMADS, and that using $q = 1.5$ with fewer samples per iteration can give better performances both for Algorithm \ref{alg:GS} and \ref{alg:DFO-TRNS}. The trust-region method also appear to show better performance when considering lower accuracy parameters. The trus-region approach seems to work better than the direct-search one, which is not surprising even considering fixed geometry model building.}

  \begin{figure}[h]
 	\centering
 	\begin{subfigure}[b]{0.34\textwidth}
 		\includegraphics[width=\linewidth]{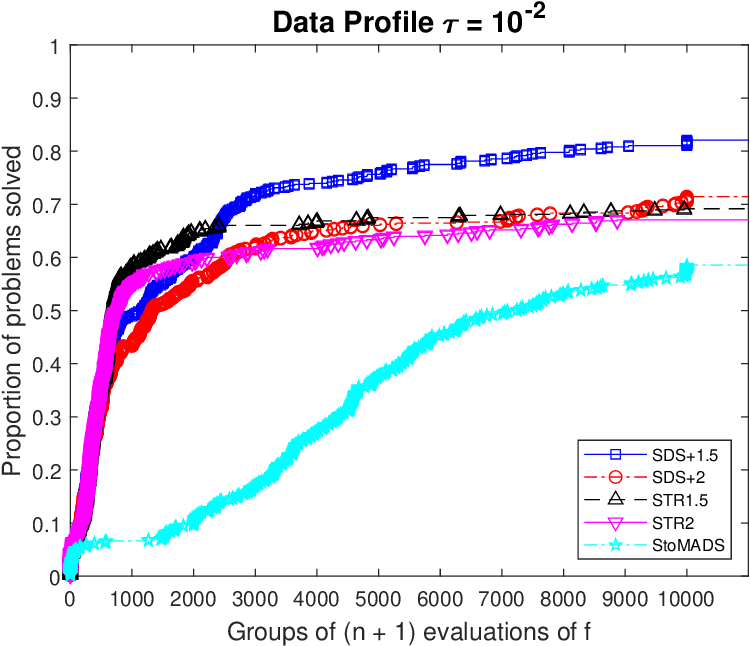}
 	\end{subfigure}	
 	\begin{subfigure}[b]{0.34\textwidth}
 		\includegraphics[width=\linewidth]{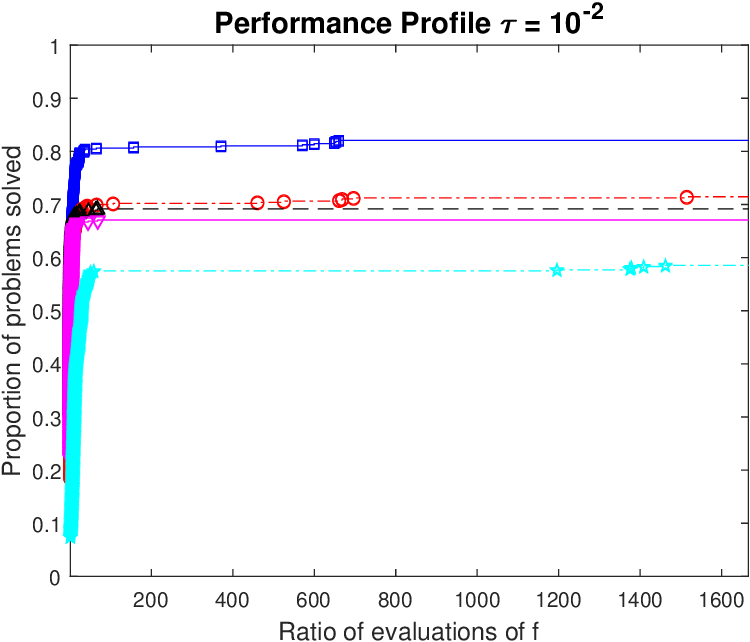}
 	\end{subfigure}
 	
 	\begin{subfigure}[b]{0.34\textwidth}
 		\includegraphics[width=\linewidth]{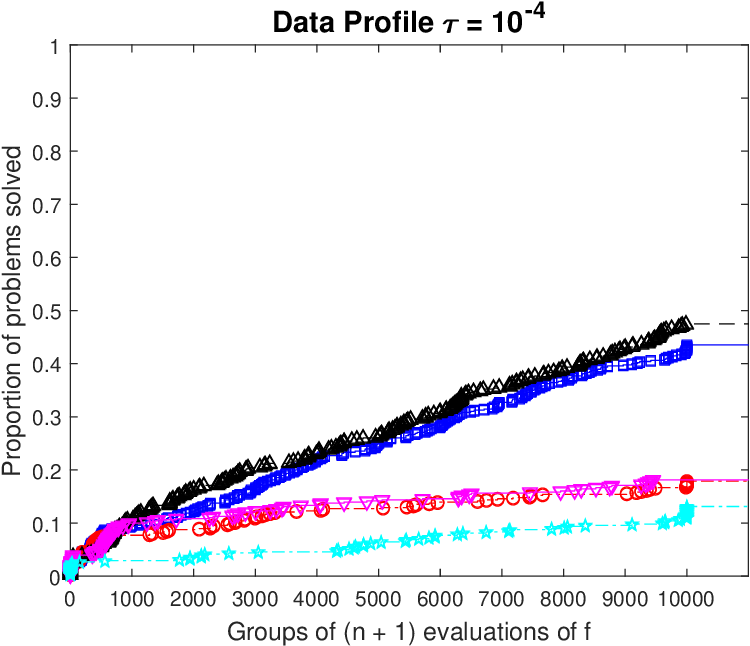}
 	\end{subfigure}
 	\begin{subfigure}[b]{0.34\textwidth}
 		\includegraphics[width=\linewidth]{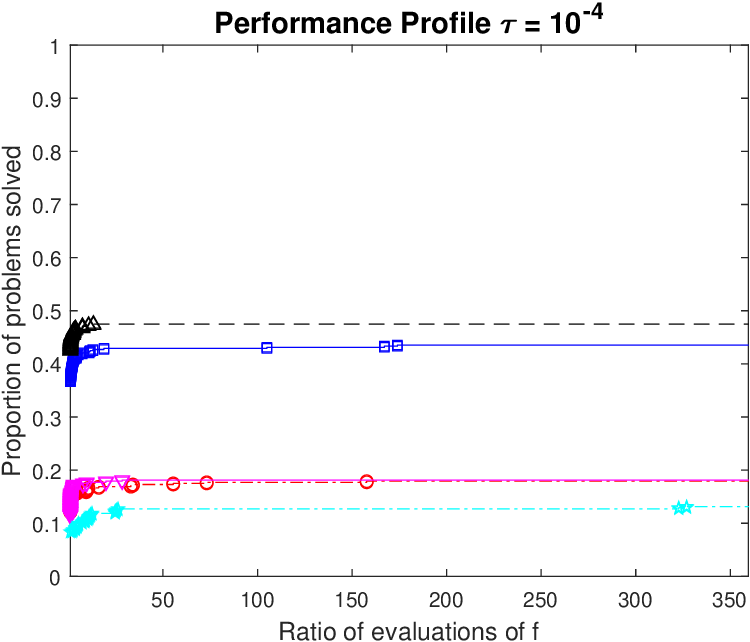}
 	\end{subfigure}
 	\caption{Data and performance profiles for Algorithm \ref{alg:GS} with $q \in \{2, 1.5\}$, Algorithm \ref{alg:DFO-TRNS} with $q \in \{2, 1.5\}$ and StoMADS.}
 	\label{fig:overall}
 \end{figure}

	\section{Concluding remarks and future work}\label{conclusions}

	This manuscript proposed a new tail bound condition for function
estimation in stochastic derivative-free optimization, provably weaker than probabilistic conditions appearing in previous works. We showed how this condition can be obtained under a finite moment assumption on the black-box noise, generalizing finite variance. This naturally led to defining a trade-off between noise moment and number of samples per iteration, generalizing the classic $O(\Delta_k^{-4})$ sample bound of the finite variance case, with improvements for higher moments. 

Our tail bound assumption allowed us to obtain convergence of both a direct-search and a trust-region method \damiano{using a reduced number of samples per iteration}. Surprisingly, unlike prior works on stochastic DFO requiring multiple probabilistic conditions for convergence, in this work a single tail bound is sufficient to prove that the sequence of stepsizes/radii converges to~0, and to conclude convergence to Clarke stationary points. 

There are a few future research developments. 
A first one is the analysis of trust-region algorithms based on non-smooth random local models under the new conditions. Possible choices of the model include piecewise linear models and random smooth functions like those used in Bayesian optimization. Studying tailored models for special cases where the objective is the non-smooth composition of smooth functions (like for instance the maximum of smooth functions) is a related challenge.
Other possible research topics include the extension of our analysis to the constrained case, its integration within global optimization schemes, and numerical tests for non-smooth versions of the trust-region scheme.


	\clearpage
	\bibliographystyle{plain}
	\bibliography{sds} 	
	\appendix
	\section{Appendix}\label{appendix}
In this appendix, we report the missing proofs and some additional numerical results.

\subsection{Proofs}

We now recall the Rosenthal inequality \cite[Equation (1)]{ibragimov2002exact}, together with a corollary useful for several results of Section \ref{s:si}. This inequality states that, if $\{Z_i\}_{i \in [1:p]}$ is a sequence of independent random variables with 0 mean and finite $r$-th moment, $r > 2$, and $S = \frac{1}{p}\sum_{i=1}^p Z_i$, one has
\begin{equation}\label{eq:rostrue}
\me\left[|S|^r\right] \leq p^{-r} C_r \max \left(\sum_{i = 1}^{p} \me\left[|Z_i|^r \right],  \left(\sum_{i = 1}^{p} \me\left[|Z_i|^2 \right]  \right)^{\frac{r}{2}} \right) \, ,
\end{equation}
for some constant $C_r > 0$ depending from $r$. \\
We report here the corollary, which concerns the special case of i.i.d. samples. 
\begin{proposition}
	If $\{Z_i\}$ is a sequence of independent copies of a random variable $Z$, $r \geq 2$ and $S$ are defined as above, then 
	 \begin{equation} \label{eq:roscor}
		\me\left[|S|^r\right] \leq C_r p^{-\frac{r}{2}}\me\left[|Z|^r\right] \, .
	\end{equation}
\end{proposition}
\begin{proof}
	For $r = 2$, the result trivially holds with $C_r = 1$, since
	\begin{equation*}
		\me\left[|S|^2\right] = p^{-2} \sum_{i = 1}^p \me\left[Z_i^2\right] = p^{-1} \me\left[Z^2\right] \, .
	\end{equation*}
Under the assumptions of this proposition \eqref{eq:rostrue} reduces to
\begin{equation} \label{eq:rosiid}
	\me\left[|S|^r\right] \leq p^{-r} C_r \max \left(p\me\left[|Z|^r\right] , p^{\frac{r}{2}}\me\left[|Z|^2 \right]^{\frac{r}{2}} \right) \, .
\end{equation}
Now,
\begin{equation} \label{eq:rosint}
	\begin{aligned}
		& p^{-r} C_r \max \left(p\me\left[|Z|^r\right] , p^{\frac{r}{2}}\me\left[|Z|^2  \right]^{\frac{r}{2}} \right) \\
		& \leq C_r p^{-r}  \max \left( p \me\left[|Z|^r\right], p^{\frac{r}{2}} \me\left[|Z|^r \right] \right) \leq C_r p^{-\frac{r}{2}}\me\left[|Z|^r\right] \,,
	\end{aligned}
\end{equation}
where Jensen's inequality is used on the second argument of the max operator in the first inequality and $r\geq 2$ and $p \geq 1$ in the second inequality. By concatenating \eqref{eq:rosiid} and~\eqref{eq:rosint}, \eqref{eq:roscor} is proved.
\end{proof}


\begin{proof}[Proof of Theorem \ref{th:a1}]
	Let $\bar{F}_k = F_k - f(X_k)$ and $\bar{F}_k^g = F_k^g - f(X_k + \Delta_k G_k) $, for $F_k$ and $F_k^g$ average of $p_k$ samples, with $\xi_{k, i}$ and $\xi^g_{k, i}$ independent samples for $i \in [1:p]$:
		\begin{equation*}
			\begin{aligned}
							F_k & = \frac{1}{p_k} \sum_{i = 1}^{p_k} F(X_k, \xi_{k, i}) \\
							F_k^g & = \frac{1}{p_k} \sum_{i = 1}^{p_k} F(X_k+ \Delta_k G_k, \xi^g_{k, i}) \, .
			\end{aligned}
		\end{equation*} 
  
	We start with the case $q > 2$, implying $r \in (1, 2)$. By the conditional version of \cite[Theorem 2]{von1965inequalities}, we have
	\begin{equation}\label{eq:von}
		\me\left[|\bar{A}_k|^r \ | \ \F_{k - 1}\right] \leq 2M_r p_k^{1 - r}
	\end{equation}
	for $\bar{A}_k = \bar{F}_k, \bar{F}_k^g$. 
	Let now $A_k = \bar{F}_k - \bar{F}_k^g$. We can then prove
	\begin{equation}\label{eq:rmoment}
		\me\left[|A_k|^r \ | \ \F_{k - 1}\right] \leq 2^{r - 1}\me\left[|\bar{F}_k|^r + |\bar{F}_k^g|^r \ | \ \F_{k - 1}\right] \leq 2^{r + 1}M_r p_k^{1 - r} \, ,
	\end{equation}
	where we used $(|a| + |b|)^r \leq 2^{r - 1} (|a|^r + |b|^r)$ for $a,b \in \R$ in the first inequality, and \eqref{eq:von} in the second.
	 Applying \eqref{eq:rmoment} we obtain 
	\begin{equation}\label{eq:alpha2tb}
		\begin{aligned}
			& \Pb\left(|A_k| \geq \alpha \Delta_k^{\frac{r}{r - 1}} \ | \ \F_{k - 1}\right) = \Pb\left(|A_k|^r \geq \alpha \Delta_k^{r^2/{r - 1}} \ | \ \F_{k - 1}\right)   \\ 
			& \leq \frac{\me\left[ |A_k|^r \ | \ \F_{k - 1}\right]}{\alpha^r \Delta_k^{r^2/(r - 1)}} \leq  2^{r + 1}M_r  \frac{p_k^{1 - r}}{\alpha^r\Delta_k^{r^2/(r - 1)}} \, ,
		\end{aligned}
	\end{equation}
	where for $p_k = O(\Delta_k^{-\frac{r^2}{(r - 1)^2}}) = O(\Delta_k^{- q^2})$ the right-hand side of~\eqref{eq:alpha2tb} is $O(1/\alpha^r)$ and Assumption~\ref{ass:2} follows.
 
 	We now deal with the case $q \in (1,2]$, corresponding to $r \in [2, + \infty)$. We will apply the conditional version of~\eqref{eq:roscor} with $S=\bar{F}_k$ and
  $Z = F(X_k, \xi) - f(X_k)$, and write
\begin{equation} \label{eq:rosenthal1}
	\begin{aligned}
	 \me\left[|\bar{F}_k|^r \ | \ \F_{k - 1}\right] \leq  C_r p_k^{-\frac{r}{2}}\me\left[|Z|^r \ | \ \F_{k - 1}\right] \leq C_r M_r p_k^{-\frac{r}{2}} \, ,
 \end{aligned}
	\end{equation} 
where we used~\eqref{eq:mbound} in the second inequality. Of course \eqref{eq:rosenthal1} holds with $\bar{F}_k^g$ instead of $\bar{F}_k$ as well. Then, reasoning as in \eqref{eq:von}, we get
	\begin{equation*}
			\me\left[|A_k|^r \ | \ \F_{k - 1}\right] \leq 2^{r} C_r M_r p_k^{-\frac{r}{2}} \, ,
	\end{equation*}
	and analogously to \eqref{eq:alpha2tb}:
	\begin{equation*}
		\begin{aligned}
			& \Pb\left(|A_k| \geq \alpha \Delta_k^{\frac{r}{r-1}} \ | \ \F_{k - 1}\right) =  \Pb\left(|A_k|^r \geq \alpha^r \Delta_k^{\frac{r^2}{r - 1}} \ | \ \F_{k - 1}\right) \\
			& \leq \frac{\me\left[|A_k|^r \ | \ \F_{k - 1}\right]}{\alpha^r \Delta_k^{r^2/(r - 1)}} \leq \frac{2^r C_r M_r p_k^{-\frac{r}{2}}}{\alpha^r\Delta_k^{r^2/(r - 1)}} \, .	
		\end{aligned}
	\end{equation*}
	In particular, for $p_k = O(\Delta_{k}^{\frac{-2r}{r - 1}}) = O(\Delta_k^{-2q})$, we retrieve Assumption \ref{ass:2}. 
\end{proof}

\begin{proof}[Proof of Proposition \ref{p:gaussian}]
By setting $x = y$ in \eqref{eq:quadratickernel} we get
\begin{equation} \label{eq:var}
		\tx{Var}_\xi[F(x, \xi)] = \sigma^2 \, .
	\end{equation}
Moreover, we have
\begin{equation} \label{eq:cov}
	\tx{Cov}_\xi(F(x, \xi), F(y, \xi)) = \sigma^2\tx{exp}\left( -\frac{\n{x - y}^2}{2l^2}\right)	\geq \sigma^2(1-\frac{\n{x - y}^2}{2l^2}) \, ,
\end{equation}
where we used \eqref{eq:quadratickernel} in the equality and $e^x \geq 1 + x$ in the inequality. 
We therefore have 
\begin{equation} \label{ass:corr2}
	\begin{aligned}
		& \me_{\xi}\left[|\bar{F}(x, \xi) - \bar{F}(y, \xi)|^2\right] = \me_{\xi}\left[\bar{F}(x, \xi)^2\right] + \me_{\xi}\left[\bar{F}(y, \xi)^2\right] - 2\me_{\xi}\left[\bar{F}(x, \xi)\bar{F}(y, \xi)\right] \\
		& = \tx{Var}_{\xi}(F(x, \xi)) + \tx{Var}_{\xi}(F(y, \xi)) - 2\tx{Cov}_{\xi}(F(x, \xi), F(y, \xi)) \\
		& \leq 2\sigma^2 - 2\sigma^2(1-\frac{\n{x - y}^2}{2l^2}) = \frac{\sigma^2}{l^2}\n{x - y}^2 \, .	
	\end{aligned}
\end{equation}
where we applied \eqref{eq:var} and \eqref{eq:cov} in the last inequality. Let now $V_{x, y} = \me_{\xi}\left[|\bar{F}(x, \xi) - \bar{F}(y, \xi)|^2\right] = \tx{Var}_{\xi}\left[\bar{F}(x, \xi) - \bar{F}(y, \xi)\right]$. $(F (x, \xi), F (y, \xi))$ is a bivariate Gaussian vector since $\{(F (x, \xi)\}$ is a Gaussian process. Thus,
the linear combination $F (x, \xi)-F (y, \xi)$ is still Gaussian, whence $\bar{F}(x, \xi)-\bar{F}(y, \xi)$ is Gaussian
with mean $0$. In particular, we can write $\bar{F}(x, \xi) - \bar{F}(y, \xi) = \sqrt{V_{x, y}}N$, with $N$ having standard normal distribution. 
We conclude by noticing, for $r \geq 2$,
\begin{equation*}
	\me_{\xi}\left[|\bar{F}(x, \xi) - \bar{F}(y, \xi)|^r\right] = \me\left[V_{x, y}^{\frac{r}{2}}|N|^{r}\right] = V_{x, y}^{\frac{r}{2}} \me\left[|N|^{r}\right] = V_{x, y}^{\frac{r}{2}} \bar{M}_r 	 \leq   \frac{\sigma^r}{l^r}\n{x - y}^r \bar{M}_r \, ,
\end{equation*}
where we used \eqref{ass:corr2} in the second inequality, and $\bar{M}_r$ is the $r$-th moment of the absolute value of a normal distribution with mean $0$ and variance $1$.
\end{proof}

\begin{proof}[Proof of Theorem \ref{th:corrbound}.]
 Let $A_k$ be an estimate of the difference between the errors at the current and the tentative points obtained with $p_k$ samples:
\begin{equation*}
	A_k = \frac{1}{p_k} \sum_{i = 1}^{p_k} (\bar{F}(X_k, \xi_{k, i}) - \bar{F}(X_k+\Delta_kG_k, \xi_{k, i})) \, .
\end{equation*}
Then, for $Z = \bar{F}(X_k, \xi) - \bar{F}(X_k+\Delta_kG_k, \xi)$, and $C_r$ constant depending only on $r$
\begin{equation} \label{eq:corre2bound}
	\begin{aligned}
	\me\left[|A_k|^r \ | \ \F_{k - 1}\right] \leq C_r p_k^{-\frac{r}{2}}\me\left[|Z|^r \ | \ \F_{k - 1}\right] \leq D_rC_rp_k^{-\frac{r}{2}}\n{\Delta_{k}G_k}^{r} = D_rC_rp_k^{-\frac{r}{2}} \Delta_{k}^r  \, ,
	\end{aligned}
\end{equation}
where we used the conditional version of \eqref{eq:roscor} in the first inequality, \eqref{ass:corr} in the second inequality, and $\n{G_k} = 1$ in the equality.

We thus have
\begin{equation*}
	\begin{aligned}
		&\Pb\left(|A_k| \geq \alpha \Delta_{k}^{\frac{r}{r - 1}} \ | \ \F_{k - 1}\right) = \Pb\left(|A_k|^r \geq \alpha^r \Delta_k^{\frac{r^2}{r - 1}}  \ | \ \F_{k - 1}\right) \\
	& \leq \frac{\me\left[|A_k|^r \ | \ |F_{k-1}\right]}{\alpha^r \Delta_{k}^{r^2/(r - 1)}} \leq \frac{D_rC_r\Delta_k^{-\frac{r}{r-1}}}{p_k^{\frac{r}{2}}\alpha^r} \, ,	
	\end{aligned}
\end{equation*}
where we used the conditional Chebyshev's inequality in the first inequality, and~\eqref{eq:corre2bound} in the last inequality. Hence we obtain Assumption \ref{ass:2} for $p_k = O(\Delta_k^{-\frac{2}{r - 1}}) = O(\Delta_k^{2 - 2q})$ as desired. 
\end{proof}
	
\begin{proof}[Proof of Proposition \ref{p:2eq}.]
	First, notice that  
	\begin{equation} \label{ineq:main}
		\begin{aligned}
			& \me\left[|\FMK - \FMG- (f(X_k) - f(X_k + \Delta_k G_k))|^2 \ | \ \F_{k - 1}\right]  \\
			& \leq  2(\me\left[|\FMG- f(X_k + \Delta_k G_k)|^2 \ | \ \F_{k - 1}\right] + \me\left[|\FMK - f(X_k)|^2 \ | \ \F_{k - 1}\right]) \\
			& \leq 4 k_f^2\Delta_k^4 \, ,	
		\end{aligned}
	\end{equation}
	where we used $(a + b)^2 \leq 2(a^2 + b^2)$ for $a, b \in \R$ in the first inequality, and \eqref{eq:c1} in the second.	
	We now have 
	\begin{equation*} 
		\begin{aligned}
			&  \Pb[|\FMK - \FMG- (f(X_k) - f(X_k + \Delta_k G_k))| \geq \alpha \Delta_k^2 \ | \ \F_{k - 1}] \\
			&	=  \Pb[|\FMK - \FMG- (f(X_k) - f(X_k + \Delta_k G_k))|^2 \geq \alpha^2 \Delta_k^4 \ | \ \F_{k - 1}] \\
			&	\leq  \frac{\me\left[|\FMK - \FMG- (f(X_k) - f(X_k + \Delta_k G_k))|^2  \ | \ \F_{k - 1}\right]}{\alpha^2 \Delta_k^4} \leq \frac{4 k_f^2}{\alpha^2} \, ,		
		\end{aligned}
	\end{equation*}
	where we used the conditional Chebyshev's inequality in the first inequality, and \eqref{ineq:main} in the second inequality.	By setting $\eq = 4k_f^2$ in the above equation we obtain
	\begin{equation*}
		\Pb[|\FMK - \FMG- (f(X_k) - f(X_k + \Delta_k G_k))| \geq \alpha \Delta_k^2 \ | \ \F_{k - 1}] \leq \frac{\eq}{\alpha^2}
	\end{equation*}		
	as desired.
\end{proof}

\begin{proof}[Proof of Proposition \ref{p:2eq2}.]
Notice that \eqref{eq:asp2} is trivially satisfied for $\alpha < \sqrt{\eq}$. We then just need to deal with the case $\alpha \geq \sqrt{\eq}$. First observe that by the triangular inequality 
	\begin{equation*}
		|\FMK  - f(X_k)| + |\FMG- f(X_k + \Delta_k G_k)| \geq |\FMK - \FMG- (f(X_k) - f(X_k + \Delta_k G_k))| \, ,
	\end{equation*}
 which proves in particular \eqref{eq:pineq}.
	Let $\alpha \geq \sqrt{\eq} $ be arbitrary. For $\beta = 1 - \frac{\eq}{\alpha^2}\bar{p} \in [1 - \bar{p}, 1)$,
	\begin{equation*}
		\begin{aligned}
			&	\Pb\left(|\FMK - \FMG- (f(X_k) - f(X_k + \Delta_k G_k))| \geq \alpha \Delta_{k}^2 \ | \F_{k - 1}\right) \\
			&	=  1 - \Pb\left(|\FMK - \FMG- (f(X_k) - f(X_k + \Delta_k G_k))| < \alpha \Delta_{k}^2 \ | \F_{k - 1}\right) \\
			& \leq  1 - \Pb\left(\{|\FMK  - f(X_k)| \leq \tau_f(\beta) \Delta_k^2 \} \cap \{|\FMG- f(X_k + \Delta_k G_k)| \leq \tau_f(\beta) \Delta_k^2 \} \ | \ \F_{k - 1}\right) \\ 
			&	\leq  1-\beta = \frac{\eq}{\alpha^2} \bar{p} \leq \frac{\eq}{\alpha^2} \, ,		
		\end{aligned}
	\end{equation*}
	where the second inequality follows from \eqref{betaprob}, and we were able to apply \eqref{eq:pineq} in the first inequality since by assumption $$\tau_f(\beta) < \frac{1}{2}\sqrt{\frac{\varepsilon}{1 - \beta}}= \frac{1}{2}\sqrt{\frac{\varepsilon \alpha^2}{\eq \bar{p}}} = \frac{\alpha}{2} \, ,$$ 
	using $\eq = \frac{\varepsilon}{\bar{p}}$ in the last equality. Given that $\alpha \geq \sqrt{\eq}$ is arbitrary, this concludes the proof.
\end{proof}

\subsection{Benchmark problems}\label{sec:bench}

\textcolor{white}{T}

 \begin{table}[H]
  	\centering
  	\caption{Problems used in numerical experiments.}
  	\begin{minipage}{0.49\textwidth}
  		\centering
  		\tiny{
  			\begin{tabular}{|l|c|c|}
  				\hline
  				name         & dimension & reference              \\
  				\hline
crescent           & 2                  & \cite{more2009benchmarking} \\
cb2                & 2                  & \cite{lukvsan2000test}                                            \\
charconn1          & 2                  & \cite{liuzzi2006derivative}                                            \\
charconn2          & 2                  & \cite{liuzzi2006derivative}                                            \\
demyanov-malozemov & 2                  & \cite{liuzzi2006derivative}                                           \\
dennis-woods       & 2                  & \cite{dennis1987optimization}                                           \\
wong1              & 7                  & \cite{lukvsan2000test}      \\
wong2              & 10                 & \cite{lukvsan2000test}      \\
wong3              & 20                 & \cite{lukvsan2000test}      \\
elattar            & 6                  & \cite{lukvsan2000test}                                            \\
goffin             & 50                 & \cite{lukvsan2000test}                                            \\
hald-madsen 1      & 2                  & \cite{liuzzi2006derivative}                                          \\
lq                 & 2                  & \cite{lukvsan2000test}                                            \\
ql                 & 2                  & \cite{lukvsan2000test}                                           \\
maxl               & 20                 & \cite{lukvsan2000test}      \\
maxq               & 20                 & \cite{more2009benchmarking} \\
mifflin 1          & 2                  & \cite{karmitsa2007test}     \\
mifflin 2          & 2                  & \cite{karmitsa2007test}     \\
rosen-suzuki       & 4                  & \cite{lukvsan2000test}                                            \\
wf                 & 2                  & \cite{lukvsan2000test}                                            \\
spiral             & 2                  & \cite{lukvsan2000test}                                            \\
evd 52             & 3                  & \cite{lukvsan2000test}                                            \\
kowalik-osborne    & 4                  & \cite{lukvsan2000test}                                            \\
oet 5              & 4                  & \cite{lukvsan2000test}                                           \\
oet 6              & 4                  & \cite{lukvsan2000test}                                            \\
gamma              & 4                  & \cite{lukvsan2000test}      \\
exp                & 5                  & \cite{lukvsan2000test}      \\
pbc1               & 5                  & \cite{lukvsan2000test}                                          \\
evd61              & 6                  & \cite{lukvsan2000test}                                            \\
filter             & 9                  & \cite{lukvsan2000test}      \\
polak 2            & 10                 & \cite{lukvsan2000test}      \\
polak 3            & 11                 & \cite{lukvsan2000test}      \\
polak 6            & 4                  & \cite{lukvsan2000test}      \\

  				\hline
  			\end{tabular}
  		}
  		\label{tab:1a}
  	\end{minipage}
  	\hfill
  	\begin{minipage}{0.49\textwidth}
  		\centering
  		\tiny{
  			\begin{tabular}{|l|c|c|}
  				\hline
  				name         & dimension & reference              \\
  				\hline
watson             & 20                 & \cite{lukvsan2000test}      \\
osborne 2          & 11                 & \cite{lukvsan2000test}      \\
shor               & 5                  & \cite{lukvsan2000test}                                           \\
colville 1         & 5                  & \cite{lukvsan2000test}                                           \\
hs 78              & 5                  & \cite{lukvsan2000test}                                           \\
maxquad            & 10                 & \cite{lukvsan2000test}      \\
gill               & 10                 & \cite{lukvsan2000test}      \\
mxhilb             & 50                 & \cite{karmitsa2007test}     \\
l1hilb             & 50                 &   \cite{lukvsan2000test}                                        \\
davidon 2          & 4                  & \cite{lukvsan2000test}                                           \\
shelldual          & 15                 & \cite{lukvsan2000test}      \\
steiner 2          & 12                 & \cite{lukvsan2000test}      \\
transformer        & 6                  & \cite{lukvsan2000test}                                           \\
polak 6.10         & 1                  & \cite{lukvsan2000test}                                    \\
wild1              & 20                 & \cite{karmitsa2007test}     \\
wild2              & 20                 & \cite{karmitsa2007test}     \\
wild3              & 20                 & \cite{karmitsa2007test}     \\
wild19             & 20                 & \cite{karmitsa2007test}     \\
wild11             & 20                 & \cite{karmitsa2007test}     \\
wild16             & 20                 & \cite{karmitsa2007test}     \\
wild20             & 20                 & \cite{karmitsa2007test}     \\
wild15             & 20                 & \cite{karmitsa2007test}     \\
wild21             & 20                 & \cite{karmitsa2007test}     \\
maxq               & \{10, 20, 30, 40\} & \cite{karmitsa2007test}     \\
l1hilb             & \{10, 20, 30, 40\} & \cite{more2009benchmarking} \\
lq                 & \{10, 20, 30, 40\} & \cite{more2009benchmarking} \\
cb3                & \{10, 20, 30, 40\} & \cite{more2009benchmarking} \\
cb32               & \{10, 20, 30, 40\} & \cite{more2009benchmarking} \\
af                 & \{10, 20, 30, 40\} & \cite{more2009benchmarking} \\
brown              & \{10, 20, 30, 40\} & \cite{more2009benchmarking} \\
mifflin2           & \{10, 20, 30, 40\} & \cite{more2009benchmarking} \\
crescent           & \{10, 20, 30, 40\} & \cite{more2009benchmarking} \\
crescent2          & \{10, 20, 30, 40\} & \cite{more2009benchmarking} \\
  				\hline
  			\end{tabular}
  		}
  		\label{tab:1b}
  	\end{minipage}
  \label{tab:1}
  \end{table}

\end{document}